\UseRawInputEncoding
\documentclass[reqno,twoside,11pt]{amsart}

\usepackage{amsmath,amsfonts,calrsfs,fullpage,amssymb,color,verbatim,eucal,yfonts,mathrsfs}

\newtheorem{Theorem}{Theorem}[section]

\newtheorem{Proposition}[Theorem]{Proposition}
\newtheorem{Lemma}[Theorem]{Lemma}
\newtheorem{Corollary}[Theorem]{Corollary}

\newtheorem{Hypothesis}[Theorem]{Hypothesis}
\makeatletter
\@addtoreset{equation}{section}

\makeatother

\def\le{\left}
\def\r{\right}

\def\La{\Lambda}

\def\R{\mathbb R}
\def\N{\mathbb N}

\def\E{\mathbb E}

\def\ds{\displaystyle}

\def\e{\epsilon}
\def\vs{\vspace{.1mm}\\}
\newcommand{\esssup}{\operatorname{ess\,sup}}

\newcommand{\Tr}{\operatorname{Tr}}

\newcommand{\one}{1\mkern -4mu\mathrm{l}}

\title{\bf Smoothing effects and maximal H\"older regularity for non-autonomous   Kolmogorov equations in infinite dimension}\date{}

\author[S. Cerrai]{Sandra Cerrai}
\address{Department of Mathematics\\
University of Maryland\\ 
College Park, MD 20742, USA}
\email{cerrai@umd.edu}

\author[A. Lunardi]{Alessandra Lunardi}
\address{
Dipartimento di Scienze Matematiche, Fisiche e Informatiche\\
Universit\`a di Parma\\
Parco Area delle Scienze, 53/A\\
43124 Parma, Italy}
\email{alessandra.lunardi@unipr.it}



\begin{document}

 \begin{abstract}  
We prove   smoothing properties and optimal  Schauder type estimates for   a class of  nonautonomous evolution equations driven by time dependent  Ornstein-Uhlenbeck operators in a separable Hilbert space. 
They arise as Kolmogorov equations of linear nonautonomous stochastic differential equations with Gaussian noise. 
 \end{abstract}

 \maketitle
 
\section{Introduction}

In this paper we consider a class of backward nonautonomous initial value problems, 
\begin{equation}
\label{Kolmogorov}
\left\{\begin{array}{l}
\partial_s u(s,x) + L(s)u(s,x) = \psi(s,x), \;\;0\leq s\leq t, \;x\in X,
\\
\\
u(t)=\varphi(x), \;x\in X, 
\end{array}\right.
\end{equation}
where  $X$ is a separable Hilbert space, endowed with the scalar product $\langle\cdot,\cdot\rangle_X$ and the corresponding norm $\Vert\cdot\Vert_X$, and the operators $L(t)$ are of Ornstein-Uhlenbeck type, formally given by 
\begin{equation}
L(t)\varphi(x)=\frac12 \Tr \left( B(t)B^*(t)D^2\varphi
    (x)\right)+ \langle A(t)x+f(t),\nabla \varphi(x)\rangle_X ,\quad x\in X    .
    \label{ouoperator}
\end{equation}
We assume that the family $\{A(t):\; 0\leq t\leq T\}$ generates a strongly continuous evolution operator $\{ U(t,s)\,:\,0\leq s\leq t\leq T\}\subset \mathcal L(X)$, and that $\{B(t):\; 0\leq t\leq T\}\subset \mathcal L(X)$ 
is strongly measurable. 

For $\psi\equiv 0$, \eqref{Kolmogorov} is the Kolmogorov equation associated with the (forward) stochastic differential equation
\begin{equation}
\label{stocheq}
\left\{\begin{array}{l}
dX_t (s,x)= (A(t)X_t(s,x) + f(t))dt + B(t) dW_t, \quad s<t<T, 
\\
\\
X_s (s,x)= x, 
\end{array}\right.
\end{equation}
where $W_t$ is an $X$-valued cylindrical Wiener process, and $x\in X$. Namely, it is the equation formally satisfied by \[u(s,x):= \E\, \varphi(X_t(s,x)),\ \ \ \ 0\leq s\leq t,\] for fixed $t\in [0, T]$.  For the proof of this fact in the autonomous case we refer to \cite{DPZrosso}. 
The nonautonomous case is similar; see   \cite{DPL} and the subsequent paper \cite{K} for finite dimensional (respectively, infinite dimensional) $X$.

 As in the autonomous case, the analysis of the smoothing  and the maximal regularity properties in the linear equation \eqref{Kolmogorov}
 is of fundamental importance when dealing with  nonlinear nonautonomous equations. Actually, in this case the nonlinear problem can be formulated  as a fixed point problem by using the variation of constants formula, 
    and the regularity properties of the evolution operator allow to get the well posedness of the nonlinear equation and the regularity of its solution (see e.g. \cite{C}, where autonomous Hamilton Jacobi equations in Hilbert spaces have been studied by following this approach). Moreover the use of Kolmogorov equations  in infinite dimension turns out to be a fundamental tool when studying finer properties of the associated stochastic PDEs, such as the asymptotic properties of multi-scaled infinite dimensional systems described by stochastic partial differential equation, or the regularization effect of noise in ill-posed PDEs. To this purpose, just as  examples, we refer to \cite{CF}, where Kolmogorov equations  of elliptic type have been used to prove some averaging results for systems of slow-fast SPDEs, and the regularity of their solution has played a crucial role, or to  \cite{CDPF}, where pathwise uniqueness for stochastic reaction-diffusion equations having a H\"{o}lder drift component have been studied, again by exploiting the regularity properties of the solutions of the associated Kolmogorov equations.
 The above mentioned papers deal with autonomous systems, but we believe it is important to understand what happens in the nonautonomous infinite dimensional case, too. This setting has still to be studied and with the present paper we are going  to fill the  gap.

If $X$ is finite dimensional, for every continuous and bounded $\varphi :X\to \R$ the unique bounded solution to \eqref{Kolmogorov} for $\psi \equiv 0$  is the function 
\[u(s, x) := P_{s,t}\varphi(x),\] where, for every $t \in\,[0,T]$, the evolution family $(P_{s,t})_{s \in\,[0,t]} $ is  defined by  
\begin{equation}
\label{P(s,t)}
P_{t,t} = I; \quad P_{s,t}\,\varphi(x) =  
\int_{X}\varphi(y)\,\mathcal{N} _{m^x(t,s),Q(t,s)}(dy) = \int_{X}\varphi(y+m^x(t,s))\,\mathcal{N} _{0,Q(t,s)}(dy),\ \ \  s\leq t, 
\end{equation}
with
\begin{equation}
\label{g,Q}
\begin{array}{l}
\ds{ Q(t,s):=\int_s^t 
U(t,r)B(r)B^*(r)U^*(t,r)dr , \ \ \ \ 0\leq s\leq t\leq T,}
\end{array}
\end{equation}
and
\begin{equation}
\label{m(t,s)}
\ds{g(t,s) := \int_s^tU(t,r)f(r)dr,}\ \ \ \ m^x(t,s) := U(t,s)x+g(t,s), \ \ \ \ \ 0\leq s\leq t\leq T. 
\end{equation}
(Here, as usual,  $\mathcal{N} _{m,Q}$ denotes the Gaussian measure with mean $m$ and covariance $Q$). For all details, see e.g. \cite{DPL}.

In infinite dimension   we need that the operators $Q(t,s)$ given by   \eqref{g,Q} have finite trace for $t>s$, for the Gaussian measures $\mathcal{N} _{m^x(t,s),Q(t,s)}$ be well defined. In this case    our standing assumption is   
\begin{equation}
\label{as1}
U(t,s)(X)\subseteq Q^{1/2}(t,s)(X),\ \ \ \ 0\leq s<t \leq T,
\end{equation}
 and we define 
\begin{equation}
\label{Lambda(t,s)}
 \Lambda(t,s):=Q^{-1/2}(t,s)\,U(t,s),\ \ \ \ 0\leq s<t \leq T, 
\end{equation}
where $Q^{-1/2}(t,s)$ is the pseudo-inverse of $Q^{1/2}(t,s)$. Under this assumption, we prove that, for $0\leq s<t \leq T$, the operator $P_{s,t} $ still defined by \eqref{P(s,t)}
 maps $B_b(X)$ (the space of the bounded Borel measurable functions from $X$ to $\R$) into $C^k_b(X)$ (the subspace of $k$-times Fr\`echet differentiable functions having bounded Fr\`echet  derivatives) for every $k\in \N$, and we give a representation formula for  the Fr\'echet derivatives of $P_{s,t}\varphi $ of any order, that involves the operators $\Lambda(t,s)$. 

It comes out that $\Lambda(t,s) \in \mathcal L(X)$, for $s<t$, but $ \|\Lambda (t,s)\|_{\mathcal{L}(X)}$ blows up as $t-s\to 0$. If the blow up is powerlike, i.e. if there are $\theta >0$ and $c>0$ such that
\begin{equation}
\label{as40}
\|\Lambda(t,s)\|_{\mathcal{L}(X)}\leq c\,(t-s)^{-\theta},\ \ \ \  0\leq s<t \leq T,
\end{equation}
we prove   maximal H\"older regularity results for the mild solution to \eqref{Kolmogorov}, namely the function $u$ given  by the representation formula 
\begin{equation}
\label{u}
u(s,x) = P_{s,t}\varphi(x) - \int_s^t (P_{s,\sigma }\psi (\sigma, \cdot))(x)d\sigma, \quad  0\leq s\leq t,\ \  \; x\in X. 
\end{equation}

Notice that if $X$ is finite dimensional, \eqref{as40} holds with $\theta = 1/2$ whenever the operators $L(s)$ in \eqref{ouoperator}  are uniformly elliptic. In infinite dimension, for every $\theta \geq 1/2$ we exhibit examples (with  unbounded operators $A(t)$) such that  \eqref{as40} holds, see Section \ref{Sect:Examples}.

The smoothing properties of $P_{s,t}$ are in Section 2. We prove that $P_{s,t}$ maps $B_b(X)$ into $C^n_b(X)$ for $\leq s<t\leq T$ and for every $n\in \N$, and 
there are constants $C_n>0$ such that 
$$\sup_{x \in\,X}\,\|D^n(P_{s,t}\varphi)(x)\|_{\mathcal{L}^n(X)}\leq C_n\, \|\varphi\|_{\infty}\,\|\La(t,s)\|^n_{\mathcal{L}(X)}, \quad \varphi\in B_b(X). $$

In Section 3 we prove 
maximal H\"older regularity results, in the case that \eqref{as40} holds.   More precisely, we prove that if $\alpha\geq 0$ and $\alpha + 1/\theta$ is not integer, then for every 
  $\varphi \in C^{\alpha + 1/\theta}_b(X)$ and $\psi \in C_b([s,t]\times X)$  such that 
  \[\sup_{\sigma \in [s,t]} [\psi(\sigma, \cdot)]_{C^{\alpha}(X; \R)} <+\infty,\] 
 the mild solution   $u(s, \cdot)$ belongs to $C^{\alpha + 1/\theta}_b(X )$ and 
  $$\sup_{0\leq s\leq t}\|u(s, \cdot)\|_{C^{\alpha + 1/\theta}_b(X )} \leq c\,(\|\varphi\|_{C^{\alpha + 1/\theta}_b(X)} + \sup_{\sigma \in [s,t]} [\psi(\sigma, \cdot)]_{C^{\alpha}(X)} ),$$
  for some constant  $c>0$ independent of $s$, $t$, $\varphi$, $\psi$.  {(In the case $\alpha =0$ we just need $\psi \in C_b([s,t]\times X)$).
If instead $\alpha + 1/\theta$ is integer, $u(s, \cdot)$ is not $\alpha + 1/\theta$ times differentiable in general,  it only belongs to the Zygmund space 
$Z^{\alpha + 1/\theta}_b(X)$ for every $s$. This phenomenon is not due to the infinite dimension of $X$ nor to the time dependence of the data; for instance if $X=\R^d$ and $A(t)\equiv 0$, $B(t) \equiv \sqrt{2}I$ we have $L(t) \equiv \Delta$, $\theta = 1/2$, and  it is well known that the  solutions  of the  heat equation $\partial_s u + \Delta u = \psi$ do not possess second order space derivatives for  general $\psi\in C_b([0,T]\times \R^d)$,  unless $d=1$. 

However, while in finite dimension  the function $u$ defined by \eqref{u} is a classical solution to \eqref{Kolmogorov} (even for a larger class of equations, see e.g.  \cite{LL}), in infinite dimension it is not differentiable with respect to $s$ in general. 

 In the autonomous case $A(t)\equiv A$, $B(t)\equiv B$, $f(t)\equiv 0$, the smoothing properties of the transition semigroup \eqref{stocheq} under assumption \eqref{as1} are well known, see e.g. \cite[Ch. 6]{DPZbrutto}. In finite dimension, Schauder estimates for stationary and evolution Ornstein-Uhlenbeck equations were proved in \cite{DPL0,L}. In infinite dimension, 
Schauder  theorems  for stationary equations were first proved in \cite{CDP} 
under assumption \eqref{as1} in the case that $B=I$ and $A$ is the generator of an analytic semigroup of negative type, where   \eqref{as40} holds  with $\theta=1/2$. See also \cite[Ch. 5]{C}, and \cite{cdp}. More recently, in \cite{LR}  Schauder and Zygmund type theorems were proved both for stationary and evolution equations under assumptions that extend \eqref{as1} and   \eqref{as40}.

The paper ends up with  two examples of classes of (genuinely nonautonomous) families of operators $A(t)$, $B(t)$ that satisfy our assumptions. In the first one $X$ is any infinite dimensional separable Hilbert space, and we consider the case where $A(t)$, $B(t)$ are diagonal operators with respect to the same Hilbert basis of $X$. Although this is a bit ``artificial", it allows to prove that for every $\theta \geq 1/2$ there are examples such that  \eqref{as40} holds with exponent $\theta$ and does not hold with any exponent $\theta'<\theta$. 

In the second example the operators $A(t)$ are the realizations of  second order elliptic differential operators in $X= L^2(\mathcal O)$ with Dirichlet boundary conditions and smooth enough coefficients, where $\mathcal O \subset \R^d$  is a bounded smooth open set, and $B(t)\in \mathcal L( L^2(\mathcal O),  L^q(\mathcal O))$ for suitable $q\geq 2$. 
As in the autonomous case, we can take $B(t)\equiv I$ if $d=1$.

 \subsection{Notations and assumptions}
 
If $X$ and $Y$ are Banach spaces, by $\mathcal L(X)$ and   $\mathcal L(X, Y)$ we denote the spaces of the linear bounded operators from $X$ to $X$ and  from $X$ to $Y$, respectively. Moreover, we denote by $\mathcal L^+_1(X)$ the subset of $\mathcal L(X)$ consisting of all non-negative and symmetric operators having finite trace.

$ B_b(X; Y)$ and $C_b(X; Y)$ denote the space of all bounded Borel measurable (resp. bounded  continuous)   functions $F:X\mapsto Y$, endowed with the sup norm 
\[\|F \|_{\infty}:= \sup_{x\in X} \|F(x)\|_Y.\] If $X=\R$ we set $B_b(X; \R) =  B_b(X)$ and  $C_b(X; \R):= C_b(X)$. 
For every $n\in \N$, $C^n_b(X)$ denotes the subspace of $ C_b(X)$ consisting of the $n$ times  Fr\'echet differentiable functions $f$ from $X$ to $\R$, having bounded 
Frech\'et derivatives up to the order $n$, endowed with the norm
$$\|f\|_{C^n_b(X)} = \|f\|_{\infty} + \sum_{k=1}^n \sup_{x\in X}\|D^k f(x)\|_{ {\mathcal L}^{k}(X)}, $$
where, for every $k$-linear continuous function $T:X^k\mapsto \R$, we denoted
$$\|T\|_{ {\mathcal L}^{k}(X)} := \sup\bigg\{ \frac{|T(h_1, \ldots, h_k)|}{\|h_1\|_X\cdots \|h_k\|_X}: \; h_i\in X\setminus \{0\} \bigg\}. $$

For $\alpha\in (0, 1)$, we set as usual
$$C^{\alpha}(X; Y) := \bigg\{F\in C_b(X; Y):  \, [F]_{C^{\alpha}(X, Y)} := \sup_{x\in X, \,h\in X\setminus\{0\}}\frac{\|F(x+h) - F(x)\|_Y}{\|h\|_X^{\alpha }} <+\infty \bigg\}, $$
and we endow $C^{\alpha}(X; Y)$ with the norm
$$\|F\|_{C^{\alpha}_H(X; Y)} :=  \|F\|_{\infty} + [F]_{C^{\alpha}(X, Y)}.$$

The Zygmund space $Z^1_b(X, Y)$ is defined by 
$$\begin{array}{l}
Z^1_b(X, Y) := 
\\
\\
\ds \bigg\{F\in C_b(X; Y):  \, [F]_{Z^1(X, Y)} := \sup_{x\in X, \,h\in X\setminus \{0\}}\frac{\| F(x+2h)-2F(x+h) + F(x)\|_Y}{\|h\|_X} <+\infty \bigg\} \end{array}.$$
We endow it with the norm
$$\|F\|_{Z^1_b(X, Y)} :=  \sup_{x\in X}\|F(x)\|_Y + [F]_{Z^1(X, Y)} .  $$
If $Y=\R$, we set $C^{n}(X; \R) =: C^{n}(X)$, for $n\in \N$, $C^{\alpha}(X; \R) =: C^{\alpha}(X)$, for $\alpha\in (0,1) $, and $Z^1(X; \R) =: Z^1(X)$. 

Higher order H\"older and Zygmund spaces of real valued functions  are accordingly defined,  as follows. 
 For $\alpha\in (0, 1)$ and $n\in \N$, we set
\begin{equation}
\label{def:higherHolder}
\begin{array}{c}
C^{\alpha +n}_b(X) := \{ f\in C^n_b(X): \: D^nf \in C^{\alpha}(X, \mathcal L^n(X))\}, 
\\
\\
\|f\|_{C^{\alpha +n}_b(X)} := \|f\|_{C^{n}_b(X)} + [D^nf]_{C^{\alpha}(X, \mathcal L^n(X))},
\end{array}
\end{equation}
and  for $n\in \N$, $n\geq 2$, 
\begin{equation}
\label{def:higherZygmund}
\begin{array}{c}
Z^{n}_b(X) := \{ f\in C^{n-1}_b(X): \: D^{n-1}f \in Z^{1}_b(X, \mathcal L^{n-1}(X))\}, 
\\
\\
\|f\|_{Z^{ n}(X)} := \|f\|_{C^{n-1}_b(X):} + [D^{n-1}f]_{Z^{1}(X, \mathcal L^{n-1}(X))}.
\end{array}
\end{equation}

Throughout the paper we shall
denote
\[\Delta = \{(s,t)\in [0,T]^2:\; 0\leq s < t\leq T\}.\]
Moreover, we shall assume the following conditions on the evolution operator $U(t,s)$, the operators $B(r)$ and the function $f$.

\begin{Hypothesis}
	\label{H1}
\begin{itemize}
\item[(i)]	$\{U(t,s):\; (s,t) \in\,\bar{\Delta}\}\subset \mathcal L(X)$ is a strongly continuous evolution operator, namely 
for every $x\in X$ the function 
\[(s,t) \in\,\bar{\Delta}\mapsto U(t,s)x \in\,X,\] is continuous, and 
$$U(t,t) = I, \quad U(t,s)U(s,r) = U(t,r), \quad 0\leq r\leq s\leq t\leq T.  $$
\item[(ii)]
 The family of linear operators $\{ B(t):\; 0\leq t\leq T\}\subset \mathcal L(X)$ is bounded , and the mapping $t\mapsto B(t)x$ is measurable, for every $x\in X$.
\item[(iii)] The  function $f: [0,T]\mapsto X$ is bounded and measurable. 	
\item[(iv)]
For every $(s,t) \in\,\bar{\Delta}$, the operator $Q(t,s)$ belongs to $\mathcal L^+_1(X)$.
\end{itemize}
\end{Hypothesis}

In particular, from  Hypothesis \ref{H1}(i) it follows that there exists $M>0$ such that 
$$\|U(t,s)\|_{\mathcal L(X)} \leq M, \quad (s,t) \in\,\bar{\Delta}.  $$

The next condition will play a fundamental role when proving the smoothing effect of the evolution family $(P_{s,t})_{0\leq s\leq t}$. 

\begin{Hypothesis}
	\label{H2} For every $0\leq s<t \leq T$, we have
	\[U(t,s)(X)\subseteq Q^{1/2}(t,s)(X).\] 
\end{Hypothesis}

It is convenient to rewrite Hypothesis \ref{H2} in an equivalent way, as in the autonomous case (e.g., \cite[Appendix B]{DPZbrutto}). 

For fixed $0\leq s<t\leq T$ we consider the operator $L: L^2((s,t);X)\mapsto X$  defined by
\begin{equation}
\label{eq:L}
Ly := \int_s^t U(t,\sigma) B(\sigma)y(\sigma) d\sigma, \quad y\in L^2((s,t);X). 
\end{equation}
The adjoint operator $L ^\star :X\mapsto L^2((s,t);X)$ satisfies $\langle Ly, x\rangle _X= \langle y, L^\star x\rangle_{L^2((s,t);X)} $ for every $x\in X$, $y\in L^2((s,t);X)$, which means 
$$\int_s^t \langle U(t,\sigma )B(\sigma) y(\sigma), x\rangle _X d\sigma = \int_s^t \langle y(\sigma), L^\star x\rangle_X d\sigma, \quad x\in X, \; y\in L^2((s,t);X), $$
so that 
\[(L^\star x)(\sigma) = B^\star(\sigma) U^\star (t,\sigma) x,\ \ \ \ x\in X,\ \ a.e.\ \sigma \in (s, t),\]
 and we get 
$$LL^\star x = \int_s^t U(t,\sigma) B(\sigma) B^\star (\sigma) U^\star (t,\sigma)x\, d\sigma = Q(t,s)x, \quad x\in X. $$
Therefore, by the general theory of linear operators in Hilbert spaces (e.g. \cite[Cor. B3]{DPZrosso}) we have
\begin{equation}
\label{equiv_controlli}
\left\{ 
\begin{array}{ll}
\text{(i)} & \text{Range} \; L = \text{Range}\, ( Q(t,s)^{1/2})
\\
\\
\text{(ii)} & \| Q(t,s)^{-1/2}x \| _X = \| L^{-1}x\|_{L^2((s,t);X)}, \;\;x\in X. 
\end{array}\right. 
\end{equation}
(Here, $L^{-1}$ too is meant as the pseudo-inverse of $L$). The range of $L$ is the set of the traces at time $t$ of the mild solutions of the evolution problems
$$\left\{
\begin{array}{l}
u'(r) = A(r)u(r) + B(r)y(r), \quad s<r<t, 
\\
\\
u(s) =0, 
\end{array}\right. $$
where $y$ varies in $L^2((s,t); X)$. So, Hypothesis \ref{H2} may be reformulated requiring that 
 $U(t,s)$ maps $X$ in the trace space, for every $0\leq s<t\leq T$. Such a space may be characterized in a number of examples, see Section \ref{Sect:Examples}. 

\hspace{2mm}

 Finally, the following last condition will be crucial in the proof of  maximal H\"older regularity in the Kolmogorov equation \eqref{Kolmogorov}.

  \begin{Hypothesis}
 \label{Hyp:Lambda}
There exist $C$ and $\theta >0$ such that
\begin{equation}
\label{as40bis}
\|\La(t,s)\|_{\mathcal{L}(X)}\leq C(t-s)^{-\theta },\quad 0\leq s<t\leq T, 
\end{equation}
where $\La(t,s)$ is the operator defined in \eqref{Lambda(t,s)}. 
\end{Hypothesis}

\medskip

Dealing with evolution equations, in what follows we shall introduce spaces of functions depending both on time and space variables. For every $a<b\in \R$ and $\alpha> 0$, we denote by $C^{0, \alpha}_b([a,b]\times X)$ the space of all bounded continuous functions $\psi :[a,b]\times X\mapsto \R$ such that $\psi(s, \cdot)\in C^{\alpha}_b(X)$, for every $s\in [a,b]$, with 
$$\|\psi\|_{C^{0,\alpha}_b ([a,b]\times X)} := \sup_{s\in [a,b]} \|\psi (s, \cdot)\|_{C^{\alpha}_b(X)}  < +\infty, $$
and, when $\alpha \geq 1$, such that  the functions  $(s,x)\mapsto (\partial ^{k}\psi /\partial h_1 \ldots \partial h_k)(s,x)$ are continuous in $[a,b]    \times X$, for every $(h_1, \ldots, h_k)\in X^k$, with $k\leq [\alpha]$. 

Next, for every $k\in \N$ we denote by $Z^{0,k}_b ([a,b]\times X)$ the space of all bounded continuous functions $\psi:[a,b]\times X\mapsto \R$ such that $\psi (s, \cdot)\in Z^{k}_b(X)$, for every $s\in [a,b]$, with 
$$\|\psi \|_{Z^{0, k}_b ([a,b]\times X)} := \sup_{s\in [a,b]} \|\psi (s, \cdot)\|_{Z^{k}_b(X)}   < +\infty, $$
and, when $k\geq 2$, such that $\psi \in C^{0,k-1}_b ([a,b]\times X)$.

\section{The evolution family $P_{s,t}$}

Due to Hypothesis \ref{H1}, for every $x \in\,X$ and $(s,t) \in\,\Delta$ the Gaussian measure  $\mathcal N_{m^x(t,s),Q(t,s)}$ in $X$ with mean $m^x(t,s)$ and covariance operator $Q(t,s)$  is well-defined. Hence, the evolution family $P_{s,t}$ defined by 
\[P_{s,t}\varphi(x)=\int_X\varphi(y)\,\mathcal N_{m^x(t,s),Q(t,s)}\,(dy),\]
for every $x \in\,X$ and $(s,t) \in\,\Delta$ and for every $\varphi \in\,B_b(X)$, is well-defined. 

We would like to stress the fact that the evolution family $P_{s,t}$ is not strongly continuous, even if $X$ is finite dimensional. However, we will prove some weaker regularity property of $P_{s,t}$ that will be enough for our purposes.
In what follows, we will denote by $\{e_j\}_{j \in\,\mathbb{N}}$ a complete orthonormal system in $X$. Moreover, for every $n \in\,\mathbb{N}$, we will denote by $\Pi_n$ the projection of $X$ onto the finite dimensional space generated by $\{e_1,\ldots,e_N\}$. 

\begin{Lemma}
\label{lemma1}
For every $N \in\,\mathbb{N}$, we denote 
\[Q_N(t,s)=\Pi_N\, Q(t,s)\, \Pi_N,\ \ \ \ \ (s,t) \in\,\bar{\Delta}.\]
Then, under Hypothesis \ref{H1}  the following maps
\begin{equation}
\label{as1000}	
(s,t) \in\,\bar{\Delta}\mapsto Q_N(t,s) \in\,\mathcal L(X),
\end{equation}
and
\begin{equation}
\label{as1001}
(s,t) \in\,\bar{\Delta}\mapsto \mbox{{\em Tr}}\,Q_N(t,s) \in\,\mathbb R,
\end{equation}
are both continuous.
Moreover, for every $(s,t) \in\,\bar{\Delta}$ we have
\begin{equation}
\label{as1002}
\lim_{N\to\infty} \Vert Q(t,s)-Q_N(t,s)\Vert_{\mathcal L(X)}=0,	
\end{equation}
and
\begin{equation}
\label{as1003}
\lim_{N\to\infty} 	\mbox{{\em Tr}}\,Q_N(t,s)=\mbox{{\em Tr}}\,Q(t,s).
\end{equation}
\end{Lemma}

\begin{proof}
For simplicity of notations, we denote
\[G(t,r)=U(t,r)B(r)B^\star(r)U^\star (t,r),\ \ \ \ G_N(t,r)=\Pi_N\, G(t,r)\, \Pi_N,\ \ \ \ (r,t) \in\,\bar{\Delta}.\]
Since $\{U(t,s)\,:\,(s,t) \in\,\bar{\Delta}\}$ is a strongly continuous evolution operator and the mapping $r \in\,[0,T]\mapsto B(r) \in\,\mathcal L(X)$ is bounded, for every fixed $r \in\,[s,T]$ the mapping
\begin{equation}
\label{as1004}	
t \in\,[r,T]\mapsto G_N(t,r) \in\,\mathcal L(X)
\end{equation}
is continuous and bounded.

Now, let $\{(s_k,t_k)\}_{k \in\,\mathbb N}$ be any sequence in $\bar{\Delta}$ converging to $(s,t)$, as $k\to\infty$. For every $x \in\,X$ we have
\[\begin{array}{l}
\ds{Q_N(t_k,s_k)-Q_N(t,s)=\int_{s_k}^{t_k} G_N(t_k,r)x\,dr-\int_{s}^{t} G_N(t,r)x\,dr	}\\[18pt]
\ds{=\int_{s_k}^{s} G_N(t_k,r)x\,dr+\int_{s}^{t} \left[G_N(t_k,r)x-G_N(t,r)x\right]\,dr+\int_{t}^{t_k} G_N(t_k,r)x\,dr=:\sum_{j=1}^3 J^N_{j,k}\,x.}
\end{array}\]
Due to the  boundedness of mapping \eqref{as1004}, we have
\[\Vert J^N_{1,k}\,x\Vert_X+\Vert J^N_{3,k}\,x\Vert_X\leq \left(|s-s_k|+|t-t_k|\right)\sup_{(t,r) \in\,\bar{\Delta}}\Vert G_N(t,r)\Vert_{\mathcal L(X)}\Vert x\Vert_X,
\]
so that
\begin{equation}
\label{as1005}
\lim_{k\to\infty 	}\Vert J^N_{1,k}\Vert_{\mathcal L(X)}+\Vert J^N_{3,k}\Vert_{\mathcal L(X)}=0.
\end{equation}
Moreover
\[\Vert J^N_{2,k}x\Vert_X\leq \int_s^t \Vert G_N(t_k,r)-G_N(t,r)\Vert_{\mathcal L(X)}\,dr \Vert x\Vert_X,\]
so that, due to the continuity and the boundedness of mapping \eqref{as1004}, thank to the dominated convergence theorem we conclude that
\[\lim_{k\to\infty 	}\Vert J^N_{2,k}\Vert_{\mathcal L(X)}=0.\]
This, together with \eqref{as1005}, allows to conclude that the mapping \eqref{as1000} is continuous.

Next, concerning the continuity of mapping 
\eqref{as1001}, if $\{(s_k,t_k)\}_{k \in\,\mathbb N}$ is any sequence in $\bar{\Delta}$ converging to $(s,t)$, as $k\to\infty$, we have
\[\mbox{Tr}\,Q_N(t_k,s_k)-\mbox{Tr}\,Q_N(t,s)=\sum_{j\leq N}
\langle \left[Q_N(t_k,s_k)-Q_N(t,s)\right]e_j,e_j\rangle.\]
Therefore, due to the continuity of mapping \eqref{as1000}, we conclude that
\[\lim_{k\to\infty} \mbox{Tr}\,Q_N(t_k,s_k)=\mbox{Tr}\,Q_N(t,s).\]

Finally, let us prove \eqref{as1002} and \eqref{as1003}.
Concerning \eqref{as1002}, first of all we notice that for an arbitrary $A \in\,\mathcal L^+_1(X)$ we have
\[\lim_{N\to\infty}A \Pi_N=\lim_{N\to\infty}\Pi_N A=A,\ \ \ \ \text{in}\ \mathcal L(X).\]
Actually, for every $x \in\,X$
\[\begin{array}{l}
\ds{\Vert A \Pi_N x-A x\Vert^2_X=\Vert  \sum_{j>N}\langle x,e_j\rangle A e_j\Vert^2_X	\leq \Vert x\Vert_X^2 \sum_{j>N}\Vert A e_j\Vert^2_X.}
\end{array}\]
Since $A \in\,\mathcal L^+_1(X)$, we have
\[\sum_{j=1}^\infty\Vert A e_j\Vert^2_X<\infty\]
and then we can conclude
\begin{equation}
\label{as1008}	\Vert A \Pi_N-A\Vert_{\mathcal L(X)}^2\leq \sum_{j>N}\Vert A e_j\Vert^2_X\to 0,\ \ \ \ \text{as}\ N\to\infty.
\end{equation}
In particular, since we are assuming that $Q(t,s) \in\,\mathcal L^+_1(X)$, for every $(s,t) \in\,\bar{\Delta}$, and
$$Q_N(t,s)-Q(t,s)=\Pi_N\left[Q(t,s)\Pi_N-Q(t,s)\right]+\left[\Pi_N Q(t,s)-Q(t,s)\right],$$
from \eqref{as1008}, we obtain \eqref{as1002}.

As for \eqref{as1003}, since
\begin{equation}
\label{as1009}	
 \mbox{Tr}\,Q(t,s)-\mbox{Tr}\,Q_N(t,s)=\sum_{j>N}\langle Q(t,s)e_j,e_j\rangle,\end{equation}
the fact that $Q(t,s)$ is trace-class allows to conclude its validity.
\end{proof}

\begin{Corollary}
\label{rem1}
Under Hypothesis \ref{H1}, for every fixed $t \in\,[0,T]$ the   mappings
\[
s \in\,[0,t]\mapsto Q(t,s) \in\,\mathcal L(X),\ \ \ \ 
s \in\,[0,t]\mapsto \mbox{Tr}\,Q(t,s) \in\,\mathbb R,
\]
are both continuous.
\end{Corollary}
\begin{proof}
The first assertion is an obvious consequence of the definition of $Q(t,s)$. To prove the second one, we
  show that the convergence  Tr$\,Q_N(t,s)\to$ Tr $\,Q(t,s)$ in \eqref{as1003} is uniform, for $s\in [0,t]$. Indeed, from  \eqref{as1009} we get 
$$ 0\leq \mbox{Tr}\,Q(t,s)-\mbox{Tr}\,Q_N (t,s) = \sum_{j>N}\int_s^t  \|B^\star(r)U^\star (t,r)e_j\|_X^2 \,dr \leq \sum_{j>N}\int_0^t  \|B^\star(r)U^\star (t,r)e_j\|_X^2 \,dr ,$$
and the right hand side vanishes as $N\to \infty$ since  
\[\mbox{Tr}\,Q(0,t) = \sum_{j\geq 1}\int_0^t  \|B^\star(r)U^\star (t,r)e_j\|_X^2 \,dr\] is finite by assumption. 
\end{proof}

\begin{Lemma}
\label{Le:continuity}
Under Hypothesis \ref{H1}, for every   $\varphi\in C_b(X)$, the mapping 
\begin{equation}
\label{as1011}	
(s, t,x) \in\,\overline{\Delta }\times X \mapsto P_{s,t}\varphi(x) \in\,\mathbb R,
\end{equation}
is measurable. Moreover, for every fixed $t \in\,[0,T]$ the mapping
\begin{equation}
\label{as1010}	
(s,x) \in\,[0,t]\times X\mapsto P_{s,t}\varphi(x) \in\,\mathbb R,
\end{equation}
is continuous.

If in addition the mapping
\[(s,t) \in\, \Delta \mapsto U(t,s) \in\,\mathcal L(X),\]
is continuous, then for every $\varphi \in\,C_b(X)$ the mapping 
\[(s, t,x) \in\,\overline{\Delta }\times X \mapsto P_{s,t}\varphi(x) \in\,\mathbb R,\]
is continuous.
 
\end{Lemma}
\begin{proof} As a first step, we notice  that for every $(s,t) \in\,\bar{\Delta}$ and every sequence $\{(s_k, t_k)\}_{k \in\,\mathbb N}\subset \bar{\Delta} $ such that    $(s_k, t_k)\to (s,t) $,  we have 
\[\mathcal{N}_{0,Q_N(s_k,t_k)}\rightharpoonup\begin{cases}\mathcal{N}_{0,Q_N(s,t)}, & \text{if}\ s<t\\
	\delta_0, & \text{if}\ s=t,
\end{cases}\]
 for every fixed $N \in\,\mathbb N$. This is because, thanks to Lemma \ref{lemma1},   $Q_N(t_k, s_k)\to Q_N(s,t)$ in $\mathcal L(X)$ and  $\mbox{Tr}\,Q_N(t_k,s_k)\to $ $\mbox{Tr}\,Q_N(s,t)$ (see e.g. \cite[Example 3.8.15]{Boga} to understand how the convergence of the covariance operator and of its trace  implies the weak convergence of Gaussian measures).

 Let now $(s_k, t_k)\to (s,t)$, $x_k\to x$, and $\varepsilon >0$. Since our sequence of Gaussian measures weakly converges, it is uniformly  tight. So, there exists a compact set $K\subset X$ such that 
\begin{equation}
\label{safine1}	
\mathcal{N}_{0,Q_N(s,t)}(X\setminus K) \leq \varepsilon, \quad \mathcal{N}_{0,Q_N(s_k,t_k)}(X\setminus K) \leq \varepsilon, \quad  k\in \N. 
\end{equation}

 Due to the strong continuity of the evolution operator $\{U(t,s)\,:\ (s,t) \in\,\bar{\Delta}\}$, we have   
 \[\lim_{k\to \infty} m^{x_k}(t_k,s_k) =  m^x(t,s).\] Then the set $ K + \left[\{ m^{x_k}(t_k,s_k):\; k\in \N\}\cup  \{ m^x(t,s)\}\right]$ is compact, so that there exists $k_0\in \N$ such that 
 \begin{equation}
 \label{safine}	\varphi(y + m^{x_k}(t_k,s_k))-\varphi(y+ m^x(t,s))|\leq \varepsilon,\ \ \ \ y  \in K,\ \ \ \ k\geq k_0.
 \end{equation}

Now, if we define

\[P^N_{s,t}\varphi(x)=\int_X \varphi(y)\mathcal N_{m^x(t,s),Q_N(t,s)}(dy),\]
we have
$$\begin{array}{l}
\ds |P^N_{s_k, t_k}\varphi(x_k) - P^N_{s, t}\varphi(x)| \\
\\
\ds \leq \ds  \int_K |\varphi(y+m^{x_k}(t_k,s_k)) -  \varphi(y+m^{x}(t,s))|\,\mathcal{N} _{0,Q_N(t_k,s_k)}(dy)  +
\\
\\
\ds + \left|  \int_K \varphi(y+m^{x}(t,s))\,\mathcal{N} _{0,Q_N(t_k,s_k)}(dy) - \int_K \varphi(y+m^{x}(t,s))\,\mathcal{N} _{0,Q_N(t,s)}(dy) \right|

\\
\\
\ds   + \left| \int_{X\setminus K} \varphi(y+m^{x_k}(t_k,s_k))\,\mathcal{N} _{0,Q_N(t_k,s_k)}(dy) -\int_{X\setminus K} \varphi(y+m^{x}(t,s))\,\mathcal{N} _{0,Q_N(t,s)}(dy) \right|.
\end{array}$$
Then, due to \eqref{safine1} and \eqref{safine} and the weak convergence of $\mathcal{N} _{0,Q_N(t_k,s_k)}$ to $\mathcal{N} _{0,Q_N(t,s)}$, there exists $k_1 \in\,\mathbb{N}$ such that 
for $k\geq \max\{k_0, k_1\}$, 
\[|P^N_{s_k, t_k}\varphi(x_k) - P^N_{s, t}\varphi(x)|\leq 2\,\epsilon+2\,\epsilon\,\Vert \varphi\Vert_\infty.\]
Due to the arbitrariness of $\epsilon>0$, this implies 
that the mapping 
\[(s, t, x) \in\,\overline{\Delta }\times X \mapsto P^N_{s,t}\varphi(x) \in\,\mathbb R,\]
is continuous, for every $N \in\,\mathbb N$ and $\varphi \in\,C_b(X)$.

Now, the measurability of the mapping \eqref{as1011}  follows once we notice that for every $(s,t,x) \in\,\bar{\Delta}\times X$ and $\varphi \in\,C_b(X)$
\begin{equation}
\label{as1012}
\lim_{N\to \infty} P^N_{s,t}\varphi(x)=P_{s,t}\varphi(x).
	\end{equation}
Actually, due to \eqref{as1002} and \eqref{as1003}, we have   
\[\mathcal N_{m^x(t,s),Q_N(t,s)}\rightharpoonup \mathcal N_{m^x(t,s),Q(t,s)}, \ \ \ \ \text{as}\ N\to\infty,\]
and \eqref{as1012} follows from the definition of $P_{s,t}$ and $P^N_{s,t}$.

The continuity of the mapping \eqref{as1010},
for every fixed $t \in\,[0,T]$, follows from Corollary \ref{rem1}, by adapting to the operators $P_{s,t}$ the arguments that we have used above for the operators $P^N_{s,t}$.

To prove the last assertion,  we observe that, since $U$ is continuous in $\Delta$,  the function 
\[(s, t) \in\,\overline{\Delta }  \mapsto Q(t,s)\in \mathcal L(X)\] is easily seen to be continuos. Also, the function  
\[(s, t) \in\,\overline{\Delta } \mapsto \mbox{Tr}\,Q(t,s)\in \R,\]
 is continuous, since for every $N\in \N$ the function $(s, t) \in\,\overline{\Delta }  \mapsto $ Tr$\,Q_N(t,s)\in \R$ is continuous and the same argument of Corollary \ref{rem1} shows that the 
convergence  Tr$\,Q_N(t,s)\to$  Tr$\,Q(t,s)$ is uniform in $\overline{\Delta }$. Since also 
\[(s, t,x) \in\,\overline{\Delta }\times X \mapsto m^x(t,s) \in\,\mathbb R,\]
is continuous, we have  that
\[\mathcal N_{m^{x_k}(t_k,s_k),Q(t_k,s_k)}\rightharpoonup  \mathcal N_{m^x(t,s),Q(t,s)}, \ \ \ \ \text{as}\ \ (s_k,t_k,x_k)\to (s,t,x),\] and this amounts to our claim by the definition of $P_{s,t}$. 
 
 \end{proof}

 In view of what we are going to do later, it is convenient to extend $P_{s,t} $ to  functions $\Phi$ in $B_b(X; X)$ or $B_b(X;  \mathcal{L}^k(X))$, by setting
\begin{equation}
\label{eq4}
{\mathbf P}_{s,t}\Phi(x) := 
\int_X
\Phi(y)\,\mathcal{N}_{m^x(t,s),Q(t,s)}
(dy), \ \ \ \ 0\leq s<t\leq T.
\end{equation}
Using again the definition and the Dominated Convergence Theorem, we obtain that $P_{s,t}$ maps
$C^1_b(X)$  into itself,  and
\[D( P_{s,t}\varphi)(x) = U^\star(t, s){\mathbf P}_{s,t}D \varphi(x),\ \ \ \  t > s, \; x \in\,X, 
\]
so that 
\[\sup_{x\in X} \|D (P_{s,t}\varphi)(x)\|_{X^\star }\leq M  \,\|D \varphi\|_{\infty}\leq M  \,\|\varphi\|_{C^1_b(X)},\ \ \ t>s,\ \ \varphi \in\,C^1_b(X).
\]

Iterating, we get the following lemma.

\begin{Lemma}
\label{Le:C^k_derivate}
Under Hypothesis \ref{H1},
for every $k\in \N$ and $0\leq s<t\leq T$, $P_{s,t}$ maps
$C^k_b(X)$ 
into itself,  and
\begin{equation}
\label{C^k}
D^k( P_{s,t}\varphi)(x)(h_1,\ldots,h_k) = P_{s,t}(D^k \varphi(\cdot)(U(t,s)h_1,\ldots,U(t,s)h_k))(x),\quad  t > s,
\end{equation}
for every $x, h_1,\ldots, h_k \in\,X$. In particular, it follows that for every $\varphi \in\,C^{k}_b(X)$
\begin{equation}
\label{as20}
\sup_{x \in\,X}\|D^k(P_{s,t}\varphi)(x)\|_{\mathcal{L}^k(X)}\leq \|U(t,s)\|_{\mathcal{L}(X)}^k\,\|\varphi\|_{C^k_b(X)} \leq M^k \,\|\varphi\|_{C^k_b(X)}, \quad 0\leq s\leq t\leq T.
\end{equation}

Moreover, for every fixed $t \in\,[0,T]$ and  $h_1,\ldots, h_k \in\,X$, the mapping
\begin{equation}
\label{as1111}	
(s,x) \in\,[0,t]\times X\mapsto D^k( P_{s,t}\varphi)(x)(h_1,\ldots,h_k)  \in\,\mathbb R,
\end{equation}
is continuous. 
If in addition the mapping
$(s,t) \in\, \Delta \mapsto U(t,s) \in\,\mathcal L(X)$
is continuous, then the function 
\[(s,t, x) \in\,\overline{\Delta} \times X\mapsto D^kP_{s,t}\varphi (x)(h_1, \ldots, h_k) \in\,\R,\] is continuous.

If $\varphi \in\,C^{\alpha}_b(X)$, where $\alpha = k+\sigma $, with $k\in \N \cup\{0\}$, and $\sigma \in (0, 1)$, then $P_{s,t}\varphi \in\,C^{\alpha}_b(X)$ and
\begin{equation}
\label{as20bis}
[D^{k}(P_{s,t}\varphi) ]_{C^\sigma(X, \mathcal{L}^k(X))}\leq \|U(t,s)\|_{\mathcal{L}(X)}^k\,[D^k\varphi]_{C^\sigma(X, \mathcal{L}^k(X))}\leq M^k\,[D^k\varphi]_{C^\sigma (X, \mathcal{L}^k(X))}.
\end{equation}
Similarly, if $\varphi \in\,Z^{k}_b(X)$, with $k\in \N$, then $P_{s,t}\varphi \in\,Z^{k}_b(X)$ and 
\begin{equation}
\label{as20bisZ}
[D^{k-1}(P_{s,t}\varphi) ]_{Z^1(X, \mathcal{L}^{k-1}(X))}\leq \|U(t,s)\|_{\mathcal{L}(X)}^{k-1}\,[D^{k-1}\varphi]_{Z^1(X, \mathcal{L}^{k-1}(X))}\leq M^{k-1}\,[D^{k-1}\varphi]_{Z^1(X, \mathcal{L}^{k-1}(X))}.
\end{equation}
%
%
%
%
%
%
\end{Lemma}
\begin{proof} 
Formula \eqref{C^k} follows  differentiating $k$ times the identity \[P_{s,t}\varphi (x) = \int_X \varphi ( y + m^x(t,s)) \,\mathcal{N}_{0,Q(t,s)}.\] 

\eqref{as20} is an immediate consequence of  \eqref{C^k}, as well as \eqref{as20bis} for $k\geq 1$ and \eqref{as20bisZ} for $k\geq 2$. 
Instead, \eqref{as20bis} for $k =0$ and \eqref{as20bisZ} for $k=1$ (where we denote as usual $D^0\psi = \psi$ for every $\psi \in C_b(X)$) follow immediately from the representation formula \eqref{P(s,t)}. 

 The proof of the continuity properties of the partial derivatives    is similar to the proof of the continuity of $P_{s,t}\varphi(x)$  in Lemma \ref{Le:continuity}, once \eqref{C^k} is established. 
\end{proof}

\

Lemma \ref{Le:C^k_derivate} yields the next corollary, that will be used in Section \ref{Sect:Schauder}.

\begin{Corollary}
\label{reg_omogenea}
Let Hypothesis \ref{H1} hold. Then for every $\alpha >0$ and for every $\varphi\in C^\alpha_b(X)$, $t>0$, the function 
$$
(s,x)  \in\, [0,t] \times X\mapsto u_0(s,x):= P_{s,t}\varphi (x) \in \,\R,
$$
belongs to $C^{0,\alpha}_b( [0,t] \times X)$, and there is $C= C(\alpha, t)>0$ such that 
\begin{equation}
\label{Holder-Holder}
\|u_0\|_{C^{0,\alpha}_b( [0,t] \times X)} \leq C\| \varphi\|_{C^\alpha_b(X)}. 
\end{equation}
Similarly, if $\varphi\in Z^k_b(X)$ for some $k\in \N$, the function $u_0$ belongs to $Z^{0,k}_b( [0,t] \times X)$, and there is $C= C(k, t)>0$ such that 
\begin{equation}
\label{Zygmund-Zygmund} \|u_1\|_{Z^{0,\alpha}_b( [0,t] \times X)} \leq C\| \varphi\|_{Z^k_b(X)}.
\end{equation}
\end{Corollary}

We recall that for any Gaussian measure $\mu:= \mathcal{N}_{0,Q }$, the space 
$H:= Q^{1/2} (X)$, endowed with the scalar product  
\[\langle z,x\rangle_H := \langle Q^{-1/2}z, Q^{-1/2}x\rangle,\] is the Cameron-Martin space of $\mu$. So, for every $z\in H$ there exists a Gaussian random variable $\hat{z} \in L^p(X, \mu)$, for every $p\in [1, \infty)$, with \[\|\hat{z}\|_{L^p(X, \mu)} \leq c_p\|z\|_H,\]  such that the Cameron-Martin formula 
\begin{equation}
\label{CM}
\mathcal{N}_{z,Q }(dy ) = \exp  \big(  -\frac 12 \|z\|_H^2+ \hat{z}(y)\big)\,\mathcal{N} _{0,Q }(dy)
\end{equation}
holds. 
Moreover, 
\[\hat{z} (x)= \langle z, x\rangle_H = \langle   Q^{-1/2}z, Q^{-1/2}x\rangle_X,\ \ \ \ x, z \in\,H.\]  By abuse of language we shall write \[ \hat{z} (x)=  \langle   Q^{-1/2}z, Q^{-1/2}x\rangle_X,\ \ \ z\in H,\ \ x\in X.\]  

Accordingly, Hypothesis \ref{H2} means that $U(t,s)$ maps $X$ into the Cameron-Martin space of $\mathcal{N}_{0,Q(t,s) }$; we shall use the notation \[ (U(t,s)h)\hat \;(x) =  \langle  \Lambda(t,s)h, Q(t,s)^{-1/2}x\rangle_X,\ \ \ \ h,\ x\in X,\] as well as the estimate 
\begin{equation}
\label{stimaCM} \|x\mapsto  \langle  \Lambda(t,s)h, Q(t,s)^{-1/2}x\rangle\|_{L^p(X, \mathcal{N}_{0,Q(t,s)})} \leq c_p \|\Lambda(t,s)\|_{\mathcal{L}(X)}\|h\|_X, \quad 0<s<t\leq T,
\end{equation}
for $h \in\,X$ and $1\leq p<+\infty$.

\begin{Theorem}
\label{Th:derivate}
Under Hypotheses \ref{H1} and \ref{H2}, $P_{s,t}\varphi \in\bigcap_{k\in \N}C^k_b(X)$, for all $\varphi \in\,B_b(X)$ and all $0\leq s<t\leq T$, and 
\begin{equation}
\label{as9}
\langle \nabla P_{s,t}\varphi)(x), h\rangle_X =\int_{X}\varphi(y+m^x(t,s))
\,\langle \La(t,s)h,Q^{-1/2}(t,s)y\rangle_X\,\mathcal{N} _{0,Q(t,s)}(dy), \quad h\in X. 
\end{equation}
Moreover, for every $n\geq 2$ we have
\begin{equation}
\label{as2}
\begin{array}{l}
\ds{
 D^n(P_{s,t}\varphi)(x)(h_1,\ldots,h_n)=
\int_X \varphi(m^x(t,s)+y)\,I_n(t,s)(y)(h_1,\ldots,h_n)\,\mathcal{N}_{0,Q(t,s)}(dy),}
\end{array}
\end{equation}
where
\begin{equation}
\label{as12}
\begin{array}{l}
\ds{I_n(t,s)(y)(h_1,\ldots,h_n):=\prod_{i=1}^n
\langle \Lambda(t,s) h_i,Q^{-1/2}(t,s)y\rangle_X}\\
\vs
\ds{+\sum_{s=1}^{r_n} (-1)^s \sum_{\substack{\i_1,\ldots,i_{2s}=1\\ i_{2k-1}<i_{2k}\\i_{2k-1}<i_{2k+1}}}^n\prod_{k=1}^s\langle \Lambda(t,s) h_{i_{2k-1}},\Lambda(t,s) h_{i_{2k}}\rangle_X\prod_{\substack{i_m=1\\i_m\neq  i_1,\ldots,i_{2s}}}^n
\langle \Lambda(t,s) h_{i_m},Q^{-1/2}(t,s)y\rangle_X,}\end{array}\end{equation}
and 
\[r_n=\begin{cases}
n/2,  &  \text{if}\ n\,\text{is even,}\\
(n-1)/2, &  \text{if}\ n\,\text{is odd.}
\end{cases}\]

In particular, for every $n \geq 2$ there exists $C_n>0$ such that
\begin{equation}
\label{as13}
\sup_{x \in\,X}\,\|D^n(P_{s,t}\varphi)(x)\|_{\mathcal{L}^n(X)}\leq C_n\, \|\varphi\|_{\infty}\,\|\La(t,s)\|^n_{\mathcal{L}(X)},\ \ \ \ s<t.
\end{equation}
\end{Theorem}

\begin{proof}
{\em Step 1.}
We prove that $P_{s,t}\varphi \in\,C^1_b(X)$ for every $\varphi \in\,B_b(X)$ and $0<s<t\leq T$,   and that \eqref{as9} holds. 

For every $\e \in\,\mathbb{R}$ and $x, h \in\,X$ we have
\[\begin{array}{l}
\ds{\frac 1\e\left(P_{s,t}\varphi(x+\e h)-P_{s,t}\varphi(x)\right)}
\vs\\
\ds{=\frac 1\e\left( \int_{X}\varphi(y+m^x(t,s))\,\mathcal{N} _{\e\,U(t,s)h,Q(t,s)}(dy)- \int_{X}\varphi(y+m^x(t,s))\,\mathcal{N} _{0,Q(t,s)}(dy)\right).}
\end{array}\]
By assumption  \eqref{as1},  $\e\,U(t,s)h \in\,Q^{1/2}(t,s)(X)$, so that, thanks to the Cameron-Martin formula \eqref{CM}, 
\[\mathcal{N} _{\e\,U(t,s)h,Q(t,s)}(dy)=\exp\le(-\frac 12 \e^2\Vert\La(t,s)h\Vert_X^2+\e\langle \La(t,s)h,Q^{-1/2}(t,s)y\rangle_X\r)\,\mathcal{N} _{0,Q(t,s)}(dy).\]
Therefore, setting for fixed $s<t$ and $h \in\,X$ 
\[f_\e(y):=-\frac 12 \e\,\Vert\La(t,s)h\Vert_X^2+\langle \La(t,s)h,Q^{-1/2}(t,s)y\rangle_X,\]
we get 
 $$\frac 1\e\left(P_{s,t}\varphi(x+\e h)-P_{s,t}\varphi(x)\right)
\vs\\
\ds = \int_{X}\varphi(y+m^x(t,s))
(\exp (\e f_\e(y)) -1)\,\mathcal{N} _{0,Q(t,s)}(dy).
$$
Now, 
\[\lim_{\e\to 0}f_\e(y)= \langle \La(t,s)h,Q^{-1/2}(t,s)y\rangle_X,\ \ \ \ \text{for a.e.}\  y.\]
Moreover 
\[|\exp (\e f_\e(y))-1| \leq C| f_\e(y)| (\exp | f_\e(y)| +1),\]
and
\[x\mapsto | \langle \La(t,s)h,Q^{-1/2}(t,s)y\rangle_X |   \exp(| \langle \La(t,s)h,Q^{-1/2}(t,s)y\rangle_X|) \in L^1(X, 
\mathcal{N} _{0,Q(t,s)}).\] So, by the 
Dominated Convergence theorem we obtain
\begin{equation}
\label{Gateaux}
\lim_{\e\to 0} \frac 1\e\left(P_{s,t}\varphi(x+\e h)-P_{s,t}\varphi(x)\right)=\int_{X}\varphi(y+m^x(t,s))
\,\langle \La(t,s)h,Q^{-1/2}(t,s)y\rangle_X\,\mathcal{N} _{0,Q(t,s)}(dy).
\end{equation}
Denoting by $D_G(P_{s,t}\varphi)(x)h$ the right hand side of \eqref{Gateaux}, by  \eqref{stimaCM}  we get 
\[|D_G(P_{s,t}\varphi)(x)h|\leq 
 \|\varphi\|_{\infty}\,\| \langle \La(t,s)h,Q^{-1/2}(t,s)\cdot \rangle_X\|_{L^1(X, \mathcal{N} _{0,Q(t,s)})} \leq 
 \|\varphi\|_{\infty}\,c_1  \|\La(t,s) \|_{\mathcal{L}(X)}\,\|h\|_X,\]
so that $D_G(P_{s,t}\varphi)(x) \in\,X^\star$. This implies  that $P_{s,t}\varphi$ is G\^ateaux differentiable at $x$ and that $D_G(P_{s,t}\varphi)(x)$ is its G\^ateaux derivative. The same estimate yields that   $P_{s,t}\varphi$ is Lipschitz continuous and  
\begin{equation}
\label{as8}
\le[P_{s,t}\varphi\r]_{\tiny{\text{Lip}}(X)}\leq c_1 \|\varphi\|_\infty \|\Lambda(t,s)\|_{\mathcal{L}(X)}.
\end{equation}

Next, if we show that $D_G(P_{s,t}\varphi):X\to X^\star$ is continuous we conclude that $P_{s,t}\varphi \in\,C^1_b(X)$ and that   \eqref{as9} holds. 
Fixed any $r \in\,(s,t)$, thanks to \eqref{as8} we have
\[\begin{array}{l}
\ds{\le|D_G(P_{s,t}\varphi)(x+k)h-D_G(P_{s,t}\varphi)(x)h\r|}=\le|D_G(P_{t,r}P_{s,r}\varphi)(x+k)h-D_G(P_{t,r}P_{s,r}\varphi)(x)h\r|\\
\vs
\ds{\leq \int_X |P_{s,r}\varphi(y+m^{x+k}(t,r))-P_{s,r}\varphi(y+m^x(t,r))|\,\le|\langle \La(t,r)h,Q^{-1/2}(t,r)y\rangle_X\r|\,\mathcal{N} _{0,Q(t,r)}(dy)}\\
\vs
\ds{\leq \|\La(r,s)\|_{\mathcal{L}(X)}\,\|\varphi\|_{\infty}\Vert m^{x+k}(t,r)-m^x(t,r)\Vert_X\|\La(t,r)\|_{\mathcal{L}(X)}\Vert h\Vert_X}\\
\vs
\ds{\leq \|\La(r,s)\|_{\mathcal{L}(X)}\,\|\varphi\|_{\infty}M\,e^{-\omega(t-r)}\Vert k\Vert_X\|\La(t,r)\|_{\mathcal{L}(X)}\Vert h\Vert_X.}
\end{array}\]
This implies that $D_G(P_{s,t}\varphi):X\to X^\star$ is uniformly continuous (in fact, it is Lipschitz continuous) and therefore $P_{s,t}\varphi \in\,C^1_b(X)$.

\medskip

{\em Step 2.} Now we prove that $P_{s,t}\varphi \in\,C^2_b(X)$  for every $0\leq s<t \leq T$ and $\varphi \in  B_b(X)$,  and 
\begin{equation}
\label{as10}
\begin{array}{l}
\ds{D^2(P_{s,t}\varphi)(x)(h_1,h_2)}\\
\vs
\ds{=\int_{X}\varphi(y+m^x(t,s))
\,\langle \La(t,s)h_1,Q^{-1/2}(t,s)y\rangle_X\,\langle \La(t,s)h_2,Q^{-1/2}(t,s)y\rangle_X\,\mathcal{N} _{0,Q(t,s)}(dy)}\\
\vs
\ds{-\langle \La(t,s)h_1,\La(t,s)h_2\rangle_X\,P_{s,t}\varphi(x).}
\end{array}
\end{equation}

As above, we first prove that for every $s<t$ and $h \in\,X$ the mapping
\begin{equation}
\label{as11}
DP_{s,t}\varphi(\cdot)h :X\to \mathbb{R},
\end{equation}
is G\^ateaux differentiable. Due to \eqref{as9} and to the Cameron-Martin formula, with the same notations introduced in Step 1, for every $\e \in\,\mathbb{R}$ and $k \in\,X$ we have
\[\begin{array}{l}
\ds{\frac 1\e\le[D(P_{s,t}\varphi)(x+\e k)h-D(P_{s,t}\varphi)(x)h\r]}\\
\vs
\ds{=\frac 1\e\le(\int_X\varphi(y+m^x(t,s)) \langle\La(t,s)h,Q^{-1/2}(t,s)(y-\e U(t,s)k)\rangle_X\,\mathcal{N} _{\e U(t,s)k,Q(t,s)}(dy)\r.}\\
\vs
\ds{\le.-\int_X\varphi(y+m^x(t,s)) \langle\La(t,s)h,Q^{-1/2}(t,s)(y)\rangle_X\,\mathcal{N} _{0,Q(t,s)}(dy)\r)=}\\
\vs
\ds{=\frac 1\e\,\int_X\varphi(y+m^x(t,s)) \langle\La(t,s)h,Q^{-1/2}(t,s)y\rangle_X\,\le(\exp\le(\e f_\e(y)\r)-1\r)\mathcal{N} _{0,Q(t,s)}(dy)}\\
\vs
\ds{-\int_X\varphi(y+m^x(t,s)) \langle\La(t,s)h,\La(t,s)k\rangle_X\,\exp\le(\e f_\e(y)\r)\,\mathcal{N} _{0,Q(t,s)}(dy).}
\end{array}\]
Proceeding as in Step 1, we obtain
\begin{equation}
\label{second}
\begin{array}{l}
\ds{\lim_{\e\to 0}\frac 1\e\le[D(P_{s,t}\varphi)(x+\e k)h-D(P_{s,t}\varphi)(x)h\r]}\\
\vs
\ds{=\int_X\varphi(y+m^x(t,s)) \langle\La(t,s)h,Q^{-1/2}(t,s)y\rangle_X \langle\La(t,s)k,Q^{-1/2}(t,s)y\rangle_X\,\mathcal{N} _{0,Q(t,s)}(dy)}\\
\vs
\ds{-\langle\La(t,s)h,\La(t,s)k\rangle_X\,P_{s,t}\varphi(x),}
\end{array}
\end{equation}
and then the right-hand side of \eqref{second} is the Gateaux derivative of $DP_{s,t}\varphi(\cdot)h$ at $x$. In particular, 
 \eqref{as10} holds and hence \eqref{as2}
in the case $n=2$.
Again, as in Step 1, we can show that the mapping
\[D_G(D(P_{s,t}\varphi)(\cdot)h):X\to X^\star,\]
is continuous, so that we  conclude that $P_{s,t}\varphi \in\,C^2_b(X)$ and \eqref{as10} holds.

\medskip

{\em Step 3.} We prove that $P_{s,t} \in\,C^n_b(X)$ for every $n\geq 3$,  and \eqref{as2} holds. 

We proceed by induction, assuming that $P_{s,t}\varphi \in\,C_b^{n-1}(X)$ and that formula \eqref{as2} holds for $D^{n-1}(P_{s,t}\varphi )$.
Thus, according again to the Cameron-Martin formula, for every $h_1,\ldots,h_{n-1}, h_n \in\,X$  and $\e \in \,\mathbb{R}$ we have
\[\begin{array}{l}
\ds{D^{n-1}(P_{s,t}\varphi)(x+\e h_{n})(h_1,\ldots,h_{n-1})}\\
\vs
\ds{= \int_X \varphi(m^x(t,s)+y)I_{n-1}(t,s)(y-\e U(t,s)h_n)\exp\le(\e f_\e(y)\r)\,\mathcal{N}_{0,Q(t,s)}(dy),}
\end{array}\]
where $I_{n-1}(t,s)(y)$ is defined as in \eqref{as12} and 
\[f_\e(y):=-\frac 12 \e\,\Vert\La(t,s)h_n\Vert_X^2+\langle \La(t,s)h_n,Q^{-1/2}(t,s)y\rangle_X.\]
Moreover we have
\[\begin{array}{l}
\ds{I_{n-1}(t,s)(y-\e U(t,s)h_n)=I_{n-1}(t,s)(y)}\\
\vs
\ds{-\e\,\sum_{i=1}^{n-1}
\langle \La(t,s)h_i,\La(t,s)h_n\rangle_X\prod_{\substack{j=1\\j\neq i}}^{n-1}\langle \La(t,s)h_j,Q^{-1/2}(t,s)y\rangle_X}\\
\vs
\ds{-\e\,\sum_{s=1}^{r_{n-1}} (-1)^s \sum_{\substack{\i_1,\ldots,i_{2s}=1\\ i_{2k-1}<i_{2k}\\i_{2k-1}<i_{2k+1}}}^{n-1}\prod_{k=1}^s\langle \Lambda(t,s) h_{i_{2k-1}},\Lambda(t,s) h_{i_{2k}}\rangle_X}\\
\vs
\ds{\times\sum_{\substack{i_m=1\\i_m\neq  i_1,\ldots,i_{2s}}}^{n-1} \langle \Lambda(t,s) h_{i_m},\Lambda(t,s) h_{n}\rangle_X\prod_{\substack{i_j=1\\i_j\neq  i_m, i_1,\ldots,i_{2s}}}^{n-1}
\langle \Lambda(t,s) h_{i_j},Q^{-1/2}(t,s)y\rangle_X+O(\e^2).}
\end{array}\]
Through  the same arguments used at Step 1  and Step 2, this implies that
\[\begin{array}{l}
\ds{\lim_{\e\to 0} \frac 1\e\,\le(D^{n-1}(P_{s,t}\varphi)(x+\e h_{n})(h_1,\ldots,h_{n-1})-D^{n-1}(P_{s,t}\varphi)(x)(h_1,\ldots,h_{n-1})\r)}\\
\vs
\ds{=\int_X \varphi(m^x(t,s)+y)I_{n-1}(t,s)(y)\langle \La(t,s)h_n,Q^{-1/2}(t,s)y\rangle_X \mathcal{N}_{0,Q(t,s)}(dy)}\\
\vs
\ds{-\int_X \varphi(m^x(t,s)+y)\Big[\sum_{i=1}^{n-1}
\langle \La(t,s)h_i,\La(t,s)h_n\rangle_X\prod_{\substack{j=1\\j\neq i}}^{n-1}\langle \La(t,s)h_j,Q^{-1/2}(t,s)y\rangle_X}\\
\vs
\ds{+\sum_{s=1}^{r_{n-1}} (-1)^s \sum_{\substack{\i_1,\ldots,i_{2s}=1\\ i_{2k-1}<i_{2k}\\i_{2k-1}<i_{2k+1}}}^{n-1}\prod_{k=1}^s\langle \Lambda(t,s) h_{i_{2k-1}},\Lambda(t,s) h_{i_{2k}}\rangle_X}\\
\vs
\ds{\times\sum_{\substack{i_m=1\\i_m\neq  i_1,\ldots,i_{2s}}}^{n-1} \langle \Lambda(t,s) h_{i_m},\Lambda(t,s) h_{n}\rangle_X\prod_{\substack{i_j=1\\i_j\neq  i_m, i_1,\ldots,i_{2s}}}^{n-1}
\langle \Lambda(t,s) h_{i_j},Q^{-1/2}(t,s)y\rangle_X\Big] \mathcal{N}_{0,Q(t,s)}(dy).}
\end{array}\]
Thus, if we rearrange all these terms, it is not difficult to show that the r.h.s. above coincides with the expression for $D^n(P_{s,t}\varphi)(y)(h_1,\ldots,h_n)$ given in \eqref{as2}. 
The fact that $P_{s,t}\varphi \in\,C_b^n(X)$ and that  its derivative coincides with \eqref{as2} follows now from the same arguments used in Step 1 and Step 2; estimate  \eqref{as13}  follows  from \eqref{stimaCM}. 
\end{proof}

As a consequence  of the previous theorem, we have the following result.
\begin{Corollary}
\label{Cor:1}
Under Hypotheses \ref{H1} and \ref{H2}, 
for every $n, k \geq 0$, with $k+n\geq 1$, and for every $\varphi \in\,C^k_b(X)$ and $0\leq s<t\leq T$, $P_{s,t}\varphi \in\,C^{k+n}_b(X)$ and
\begin{equation}
\label{as25}
\begin{array}{l}
\ds{D^{k+n}(P_{s,t}\varphi)(x)(h_1,\ldots,h_{n+k})}\\
\vs
\ds{=\int_X D^k\varphi(m^x(t,s)+y)(U(t,s)h_1,\ldots,U(t,s)h_k)\,\,I_n(t,s)(y)(h_{k+1},\ldots,h_{k+n})\,\mathcal{N}_{0,Q(t,s)}(dy).}
\end{array}
\end{equation}

For every $\alpha_2 \geq  \alpha_1\geq 0$ there is $C= C(\alpha_1, \alpha_2)$ such that 
\begin{equation}
\label{stima_generale}
\|P_{s,t}\varphi \|_{C^{\alpha_2}_b(X)} \leq C (\|\La(t,s)\|_{\mathcal{L}(X)}^{\alpha_2 -\alpha_1} +1)\,\|\varphi\|_{C^{\alpha_1}_b(X)},\ \ \ \ \varphi\in C^{\alpha_1}_b(X), \; 0\leq s<t\leq T. 
\end{equation}

Moreover, for every $k\in \N$, $h_1$, \ldots, $h_k\in X$, $t\in [0,T]$,  $\varphi \in C_b(X)$,   the function 
\[(s,x) \in\,[0,t) \times X \mapsto D^kP_{s,t}\varphi(x)(h_1, \ldots, h_k) \in\,\mathbb{R},\]   is continuous. 

If in addition the mapping $U$ is continuous in $\Delta$ with values in 
$\mathcal L(X)$, then for every $\varphi \in\,C_b(X)$ the function 
\[(s,t,x) \in\,\Delta\times X\mapsto D^kP_{s,t}\varphi(x)(h_1, \ldots, h_k) \in\,\mathbb{R},\] 
is continuous.
\end{Corollary}
\begin{proof}
From \eqref{as20} and \eqref{as20bis} we already know that \eqref{stima_generale} holds for $\alpha_1=\alpha_2$. So, we may assume that $\alpha_2 >  \alpha_1$. 

Formula \eqref{as25} follows combining together \eqref{C^k} and \eqref{as2}.  It yields the existence of  $C=C(k,n)$ such that 
\begin{equation}
\label{as26}
\sup_{x \in\,X}\|D^{k+n}(P_{s,t}\varphi)(x)\|_{\mathcal{L}^{k+n}(X)}\leq C\|\La(t,s)\|_{\mathcal{L}(X)}^n\,\|\varphi\|_{C^k_b(X)},\ \ \ \ 0\leq s<t\leq T, \; \varphi \in C^k_b(X), 
\end{equation}
and,  for  every $\alpha \in (0,1)$, 
\begin{equation}
\label{as27}
 \|D^{k+n}(P_{s,t}\varphi) \|_{C^\alpha (X, \mathcal{L}^{k+n}(X))}\leq C\|\La(t,s)\|_{\mathcal{L}(X)}^n\,\|\varphi\|_{C^{k+\alpha}_b(X)},\ \ \ \ 0\leq s<t\leq T, \; \varphi\in C^{k+\alpha}_b(X). 
\end{equation}
Through the obvious estimate $\sup_{\xi >0} \xi^k /(\xi^n +1) <+\infty$ for $k = 1, \ldots, n-1$, 
\eqref{as26} and \eqref{as27} imply \eqref{stima_generale} in the case $\alpha_2-\alpha_1\in \N$. If $\alpha_2-\alpha_1\notin \N$, we set $\alpha_2 = \alpha_1 + n + \sigma$, with $n\in \N \cup\{0\}$, $\sigma\in (0, 1)$, and  we use the interpolation inequality
\begin{equation}
\label{Jsigmanostro}
\| \psi\|_{C^{\alpha_2} (X )} \leq c \| \psi\|_{C^{\alpha_1+n}_b(X)}^{1-\sigma} \| \psi\|_{C^{\alpha_1+n +1}_b(X)}^\sigma , \quad \psi\in  C^{\alpha_1+n +1}_b(X)
\end{equation}
applied to $\psi =  P_{s,t}\varphi$, and estimates \eqref{stima_generale} with $\alpha_2$ replaced by $\alpha_1 +n$ and by $\alpha_1 +n +1$, respectively. 

The assertions about the continuity of the derivatives are consequences of Lemma \ref{Le:C^k_derivate} and Theorem \ref{Th:derivate}. Indeed, fixed any $t>0$ and $\varepsilon \in (0, t)$, we have $P_{s,t}\varphi = P_{s,t-\varepsilon}P_{t-\varepsilon,t}\varphi $ for $s  \leq t-\varepsilon$. Since $\psi:=P_{t-\varepsilon,t}\varphi \in C^k_b(X)$ by Theorem \ref{Th:derivate}, the function
\[(s,x) \in\,[0, t-\varepsilon]\times X\mapsto D^k P_{s,t-\varepsilon}\psi (x)(h_1, \ldots, h_k) =  D^kP_{s,t}\varphi(x)(h_1, \ldots, h_k) \in\,\mathbb{R},\] is continuous  by Lemma \ref{Le:C^k_derivate}. 
The proof of the last claim is similar. 
\end{proof}

Estimate \eqref{Jsigmanostro} should be known; a proof for $\alpha_1 =0$ is in \cite{L1}. However we were not able to find its proof in full generality and therefore we provide it.

\begin{Proposition}
\label{Le:C^1interp}
If $X$ is any Banach space and $0\leq \alpha_1 <\alpha<\alpha_2$, there is $C=C(\alpha_1, \alpha, \alpha_2)>0$ such that 
\begin{equation}
\label{Jsigma}
\| \psi\|_{C^{\alpha}_b (X )} \leq C \| \psi\|_{C^{\alpha_1 }_b(X)}^{\frac{\alpha_2 -\alpha}{\alpha_2-\alpha_1}} \| \psi\|_{C^{\alpha_2}_b(X)}^{\frac{\alpha -\alpha_1}{\alpha_2-\alpha_1}} , \quad \psi\in  C^{\alpha_2}_b(X).
\end{equation}
\end{Proposition}
\begin{proof}
The proof is in three steps. First we show that for  every Banach space $Y$ and for $\alpha_1\in [0, 1)$, $\alpha_2\in (\alpha_1, \alpha_1 +1]$  and for every $\alpha \in (\alpha_1, \alpha_2)$ there is $C>0$ such that  
\begin{equation}
\label{interp1}
\| \varphi \|_{C^{\alpha}_b (X; Y )} \leq C \| \varphi \|_{C^{\alpha_1 }_b(X;Y)}^{\frac{\alpha_2 -\alpha}{\alpha_2-\alpha_1}} \| \psi\|_{C^{\alpha_2}_b(X; Y)}^{\frac{\alpha -\alpha_1}{\alpha_2-\alpha_1}} , \quad \varphi \in  C^{\alpha_2}_b(X; Y). 
\end{equation}
As a second step, we use \eqref{interp1} to prove that \eqref{Jsigma} holds for 
$0\leq \alpha_1 <\alpha<\alpha_2\leq \alpha_1 +1$. In the third step we prove \eqref{Jsigma} in its full generality, by recurrence. 

\vspace{3mm}

\noindent {\em Step 1.}  First we consider the case  $0 \leq \alpha_1 <\alpha \leq \alpha_2\leq 1$. For every $\varphi  \in C^{\alpha_2}_b(X, Y)$  we have
$$\|\varphi  (x) - \varphi (y)\|_Y \leq \left\{ \begin{array}{ll} \Lambda_1(x,y):= [\varphi ]_{C^{\alpha_1}(X, Y)} \|x-y\|_X^{\alpha_1}
\\
\\
\Lambda_2(x,y):=[\varphi ]_{C^{\alpha_2}(X, Y)} \|x-y\|_X^{\alpha_2}
\end{array}\right. $$
with $\Lambda_1(x,y)$ replaced by $2\|\varphi \|_{\infty}$, if $\alpha_1 =0$, and $\Lambda_2(x,y)$ replaced by $\|D\varphi\|_{C_b(X, \mathcal L(X, Y))}  \|x-y\|_X$, if $\alpha_2 =1$. In any case, raising (i) to $\frac{\alpha_2 -\alpha}{\alpha_2-\alpha_1}$ and (ii) to $\frac{\alpha -\alpha_1}{\alpha_2-\alpha_1}$ we get 
$$
[\varphi ]_{C^{\alpha} (X,Y)} \leq C [\varphi ]_{C^{\alpha_1 }_b(X;Y)}^{\frac{\alpha_2 -\alpha}{\alpha_2-\alpha_1}}[\varphi ]_{C^{\alpha_2}_b(X; Y)}^{\frac{\alpha -\alpha_1}{\alpha_2-\alpha_1}},
$$
with $C=1$, if $\alpha_1>0$, and $C=2^{\frac{\alpha_2 -\alpha}{\alpha_2-\alpha_1}}<2$, if $\alpha_1=0$. Adding $\|\varphi\|_{\infty}$ to both sides we get 
\begin{equation}
\label{interp2}
\|\varphi \|_{C^{\alpha} (X,Y)} \leq 3 \|\varphi \|_{C^{\alpha_1 }_b(X;Y)}^{\frac{\alpha_2 -\alpha}{\alpha_2-\alpha_1}}\|\varphi\|_{C^{\alpha_2}_b(X; Y)}^{\frac{\alpha -\alpha_1}{\alpha_2-\alpha_1}}.
\end{equation}

Now we consider the case  $\alpha =1$ and $0<\alpha_1 < 1 <\alpha_2<2$. For every $\varphi\in C^{\alpha_2}_b(X; Y)$ and $x$, $h\in X$, with $\|h\|_X =1$, and  $t>0$ we have
$$\varphi(x+th) - \varphi(x) - D\varphi(x)(th) = \int_0^1 (D\varphi(x+t\sigma h) - D\varphi(x))(th) \,d\sigma$$
so that 
$$\| D\varphi(x)(th) \|_Y \leq  [\varphi ]_{C^{\alpha_1 }_b(X;Y)} t^{\alpha_1} + \int_0^1 \sigma^{\alpha_2-1 } d\sigma  [D\varphi ]_{C^{\alpha_2-1 }_b(X;\mathcal L(X,Y))}t^{\alpha_2} . $$
It follows 
\begin{equation}
\label{interp3}
\|D\varphi\|_{C_b  (X;\mathcal L(X,Y))}\leq  [\varphi ]_{C^{\alpha_1 }_b(X;Y)} t^{\alpha_1-1} + \frac{1}{\alpha_2}  [D\varphi ]_{C^{\alpha_2-1 }_b(X;\mathcal L(X,Y))}t^{\alpha_2 -1}, \quad t>0. 
\end{equation}
If $ [D\varphi ]_{C^{\alpha_2-1 }_b(X;\mathcal L(X,Y))} =0$, $D\varphi$ is constant so that $\varphi $ is affine, and since it is bounded it is constant. Therefore, 
$\|\varphi\|_{C^{\beta} (X,Y)} = \|\varphi\|_{C_b(X,Y)} $ for every $\beta >0$, and \eqref{interp1} obviously holds. If $ [D\varphi ]_{C^{\alpha_2-1 }_b(X;\mathcal L(X,Y))} \neq 0$, we choose 
\[t= (  [\varphi ]_{C^{\alpha_1 }_b(X;Y)}/  [D\varphi ]_{C^{\alpha_2-1 }_b(X;\mathcal L(X,Y))} )^{\frac 1{\alpha_2-\alpha_1}},\] 
so that \[[\varphi ]_{C^{\alpha_1 }_b(X;Y)} t^{\alpha_1-1}  =   [D\varphi ]_{C^{\alpha_2-1 }_b(X;\mathcal L(X,Y))}t^{\alpha_2 -1}.\]  This implies that
\begin{equation}
\label{interp4}
\|D\varphi\|_{C_b  (X;\mathcal L(X,Y))}\leq 2 [\varphi ]_{C^{\alpha_1 }_b(X;Y)}^{ \frac{ \alpha_2-1}{\alpha_2-\alpha_1}}  [D\varphi ]_{C^{\alpha_2-1 }_b(X;\mathcal L(X,Y))}^{\frac{1 -\alpha_1}{\alpha_2-\alpha_1}} 
\end{equation}
which yields
$$\|\varphi\|_{C^1_b(X,Y)} \leq 3 \|\varphi \|_{C^{\alpha_1 }_b(X;Y)}^{ \frac{\alpha_2 -1}{\alpha_2-\alpha_1}}  \|\varphi \|_{C^{\alpha_2}_b(X ,Y)}^{\frac{1-\alpha_1}{\alpha_2-\alpha_1}}. $$

Estimates \eqref{interp1} in the cases $\alpha_1 < \alpha <1<\alpha_2$ and $\alpha_1  <1<  \alpha <\alpha_2$ are obtained using the above ones. Indeed, if $\alpha\in (\alpha_1, 1)$, for every $\varphi\in C^{\alpha_2}_b(X ,Y)$  we have
$$\begin{array}{lll}
\|\varphi\|_{C^\alpha_b(X,Y)} & \leq & 3\|\varphi\|_{C^{\alpha_1}_b(X,Y)} ^{\frac{1-\alpha }{1-\alpha_1}} \|\varphi\|_{C^{1}_b(X,Y)} ^{\frac{\alpha -\alpha_1}{1-\alpha_1}}
\\
\\
& \leq & 3 \|\varphi\|_{C^{\alpha_1}_b(X,Y)} ^{\frac{1-\alpha }{1-\alpha_1}}( 3 \|\varphi\|_{C^{\alpha}_b(X,Y)}^{ \frac{\alpha_2-1}{\alpha_2-\alpha}} 
 \|\varphi \|_{C^{\alpha_2}_b(X ,Y)}^{\frac{1-\alpha }{\alpha_2-\alpha}}) ^{\frac{\alpha -\alpha_1}{1-\alpha_1}}\end{array}$$
so that
$$\|\varphi\|_{C^\alpha_b(X,Y)} ^{\frac{(\alpha_2-\alpha_1)(1-\alpha )}{(\alpha_2-\alpha)(1-\alpha_1)}} \leq 3^{ 1+ \frac{\alpha -\alpha_1}{1-\alpha_1}} 
\|\varphi\|_{C^{\alpha_1}_b(X,Y)} ^{\frac{1-\alpha }{1-\alpha_1}}  \|\varphi \|_{C^{\alpha_2}_b(X ,Y)}^{\frac{(1-\alpha )(\alpha -\alpha_1)}{(\alpha_2-\alpha) (1-\alpha_1)}}$$
which yields \eqref{interp1} in this case. The case $\alpha_1  <1<  \alpha <\alpha_2$ is similar.

Notice that taking $Y=\R$ and $\varphi =\psi$, this proves the statement for $0\leq \alpha_1 < \alpha <\alpha_2 <\alpha_1+1$.

\vspace{3mm}

\noindent{\em Step 2.} Let $0\leq \alpha_1 <\alpha<\alpha_2\leq \alpha_1 +1$. We may assume $\alpha_1\geq 1$, since the case $\alpha_1<1$ has already been considered in Step 1. We set $k := [\alpha_1]$ (the integral part of $\alpha_1$) and \[\widetilde{\alpha}_1 := \alpha_1-k,\ \ \ \widetilde{\alpha}:= \alpha -k,\ \ \ \widetilde{\alpha}_2 := \alpha_2-k.\] For every $\psi\in C^{\alpha_2}_b(X)$ we apply \eqref{interp1} to the function $\varphi:= D^k\psi$, with $Y= \mathcal L^k(X)$, obtaining
$$\|D^k \psi\|_{C^{\widetilde{\alpha}}(X,  \mathcal L^k(X))} \leq C \|D^k \psi\|_{C^{\widetilde{\alpha}_1}(X,  \mathcal L^k(X))} ^{\frac{\widetilde{\alpha}_2 -\widetilde{\alpha}}{\widetilde{\alpha}_2-\widetilde{\alpha}_1}}
\|D^k \psi\|_{C^{\widetilde{\alpha}_2}(X,  \mathcal L^k(X))} ^{\frac{\widetilde{\alpha} -\widetilde{\alpha}_1}{\widetilde{\alpha}_2-\widetilde{\alpha}_1}}
\leq C\|\psi\|_{C^{\alpha_1}_b(X)}^{\frac{\alpha_2 -\alpha}{\alpha_2-\alpha_1}} \|\psi\|_{C^{\alpha_2}_b(X)} ^{\frac{\alpha -\alpha_1}{\alpha_2-\alpha_1}}$$
 so that, adding $\|\psi\|_{C^{k-1}_b(X)}$ to both sides, we get 
$$\|  \psi \|_{C^{\alpha }_b (X)} \leq (C+1) \|\psi\|_{C^{\alpha_1}_b(X)}^{\frac{\alpha_2 -\alpha}{\alpha_2-\alpha_1}} \|\psi\|_{C^{\alpha_1}_b(X)} ^{\frac{\alpha -\alpha_1}{\alpha_2-\alpha_1}}. $$

\vspace{3mm}

\noindent{\em Step 3.} By Step 2 we know that \eqref{Jsigma} holds for  $\alpha_2-\alpha_1 <1$. We assume now that it holds for $\alpha_2-\alpha_1 <2^{k}$
 for some integer $k\geq 0$ and we prove that it holds for $\alpha_2-\alpha_1 <2^{k+1}$. 
 
First we take $\alpha = (\alpha_1 + \alpha_2)/2$, so that $\alpha - \alpha_1 <2^k$, and $\alpha_2 - \alpha<2^k$. We fix a small $\varepsilon \in (0, (\alpha_2-\alpha_1)/2)$ such that $  \alpha +\varepsilon - \alpha_1 = (\alpha_2-\alpha_1)/2 + \varepsilon <2^k$. By the recurrence assumption  there are $C_1$, $C_2$ such that 
$$\|\psi\|_{C^{\alpha }_b(X)} \leq C_1 \|\psi\|_{C^{\alpha_1 }_b(X)}^{\frac{\varepsilon}{\alpha -\alpha_1+ \varepsilon }} \|\psi\|_{C^{\alpha+ \varepsilon }_b(X)}^{\frac{\alpha-\alpha_1}{\alpha + \varepsilon -\alpha_1}}  , $$
and 
 $$ \|\psi\|_{C^{\alpha+ \varepsilon }_b(X)} \leq C_2 \|\psi\|_{C^{\alpha }_b(X)}^{\alpha_2-\frac{\alpha +\varepsilon}{\alpha_2-\alpha}}  \|\psi\|_{C^{\alpha_2}_b(X)}^{\frac{\varepsilon}{\alpha_2-\alpha}}, $$
 so that, taking into account that $\alpha_2-\alpha = \alpha-\alpha_1 = (\alpha_2-\alpha_1)/2$, 
 $$\|\psi\|_{C^{\alpha }_b(X)}^2 \leq C_1C_2^{\frac{\alpha_2-\alpha_1}{\alpha_2-\alpha_1 + 2\varepsilon} } \|\psi\|_{C^{\alpha_1 }_b(X)} \|\psi\|_{C^{\alpha_2}_b(X)}, $$
 which yields
 \begin{equation}
 \label{interp5}
 \|\psi\|_{C^{(\alpha_2-\alpha_1)/2}_b(X)} \leq C  \|\psi\|_{C^{\alpha_1 }_b(X)}^{1/2} \|\psi\|_{C^{\alpha_2}_b(X)}^{1/2}. 
 \end{equation}
The estimates for $\alpha <  (\alpha_2+\alpha_1)/2$ and for  $\alpha > (\alpha_2+\alpha_1)/2$ follow from this one and from the recurrence assumption. Indeed, 
for $\alpha <  (\alpha_2+\alpha_1)/2$ we have (since $ (\alpha_2+\alpha_1)/2 - \alpha_1 <2^k$)
$$\|\psi\|_{C^{\alpha }_b(X)} \leq C \|\psi\|_{C^{\alpha_1 }_b(X)}^{\frac{2}{\alpha_2 - \alpha_1}(\frac{\alpha_2 +\alpha_1}{2}-\alpha)} \|\psi\|_{C^{ (\alpha_2+\alpha_1)/2}_b(X)}^{\frac{2(\alpha - \alpha_1)}{\alpha_2 - \alpha_1}}, $$
and using  \eqref{interp5} we get 
$$\|\psi\|_{C^{\alpha }_b(X)} \leq C \|\psi\|_{C^{\alpha_1 }_b(X)}^{\frac{\alpha_2-\alpha}{\alpha_2 - \alpha_1}}\|\psi\|_{C^{ \frac{\alpha_2+\alpha_1}{2}}_b(X)}^{\frac{\alpha-\alpha_1}{\alpha_2 - \alpha_1}} , $$
with a different $C$. For $\alpha > (\alpha_2+\alpha_1)/2$ the procedure is similar: we estimate $\|\psi\|_{C^{\alpha }_b(X)}$ in terms of $ \|\psi\|_{C^{ (\alpha_2+\alpha_1)/2}_b(X)}$ and $ \|\psi\|_{C^{\alpha_1 }_b(X)}$, and then we use \eqref{interp5}. 
 So,  \eqref{Jsigma} holds for $\alpha_2-\alpha_1<2^{k+1}$. 
\end{proof}

\section{Schauder type theorems}
\label{Sect:Schauder}

Throughout this section we assume that in addition to Hypotheses \ref{H1} and \ref{H2}, also Hypothesis \ref{Hyp:Lambda} holds.
Under these conditions,  we get power-like blow up of the norm of 
$P_{s,t}$, stated in the next corollary.

\begin{Corollary}
\label{Cor:stimeP_st}
For every $n\in \N$ there exists $K_n >0$ such that 
\begin{equation}
\label{nth-derivatives}
\|D^nP_{s,t} \varphi (x)\|_{\mathcal L^n(X)} \leq \frac{K_n}{(t-s)^{n\theta}}\|\varphi\|_{\infty}, \quad 0\leq s<t\leq T, \; \varphi\in B_b(X). 
\end{equation}
For every $\alpha\in (0, 1)$ and $n\in \N$ there exists $K_{n,\alpha}>0$ such that 
\begin{equation}
\label{nth-derivativesHolder}
\|D^nP_{s,t} \varphi (x)\|_{\mathcal L^n(X)} \leq \frac{K_{n,\alpha}}{(t-s)^{(n-\alpha)\theta }}\|\varphi\|_{C^\alpha_b(X)}, \quad 0\leq s<t\leq T, \; \varphi\in C^\alpha_b(X). 
\end{equation}
\end{Corollary}
\begin{proof}
Estimates \eqref{nth-derivatives} and \eqref{nth-derivativesHolder} follow immediately from Theorem \ref{Th:derivate} and \eqref{stima_generale}. 
\end{proof}

Let now $t\in [0,T]$ and $\psi\in C_b([0,t]\times X)$. We split the function $u$ defined in \eqref{u} into $u(s,x) = u_0(s,x) + u_1(s,x) $, where 
\begin{equation}
\label{v0}
u_0(s,x):= P_{s,t}\varphi (x) , \quad u_1(s,x) := -\int_s^t P_{s, \sigma} \psi(\sigma, \cdot)(x) d\sigma , \quad 0\leq s \leq t, \; x\in X. 
\end{equation}

 \begin{Proposition}
 \label{Pr:nonoptimal_par}
For every $t\in [0,T]$ and $\psi \in C_b([0,t]\times X)$ the function $u_1$ defined in \eqref{v0}
is continuous, and we have 
\begin{equation}
\label{stimasup_par} 
\|u_1\|_{\infty}  \leq t \|\psi\|_{\infty}. 
\end{equation}
Moreover the following statements hold. 
 \begin{itemize}
 \item[(i)] Let $\theta <1$. For every $n\in \N$ such that $n<1/\theta$, the function $u_1$ belongs to $C^{0,n}_b([0, t]\times X)$, there exists $C >0$, independent of $\psi $ and $t$, such that 
 \begin{equation}
 \label{C^n0_par}
 \|u_1\|_{C^{0,n}_b([0, t]\times X)}\leq C\|\psi \|_{\infty}.
 \end{equation}
  \item[(ii)] Let $\alpha\in (0,1)$ be such that $\alpha + 1/\theta >1$. For every $\psi \in C^{\alpha}_b(X)$ and for every $n\in \N$ such that $n<\alpha +1/\theta$, the function  $u_1$ belongs to $ C^{0,n}_b([0,t]\times X)$, and  there exists $C >0$, independent of $\psi $ and $t$, such that 
 \begin{equation}
 \label{C^nalpha_par}
 \|u_1\|_{C^{0,n}_b([0,t]\times X)}\leq C\|\psi \|_{C^{0,\alpha}_b([0,t]\times X)}
 \end{equation}
\end{itemize}
\end{Proposition}
\begin{proof}
Fix  $s_0\in [0,t]$ and  $x$, $x_0\in X$. If $ s_0\leq s\leq t $ we have 
$$\begin{array}{lll}
|u_1(s,x) - u_1(s_0, x_0) | & \leq & \ds \int_{s}^{t} |P_{s,\sigma }\psi (\sigma, \cdot)(x) - P_{s_0, \sigma }\psi (\sigma, \cdot)(x_0)|\,d\sigma  + \int_{s_0}^s  |P_{s_0,\sigma }\psi (\sigma, \cdot)(x)|\,d\sigma 
\\
\\
& = & \ds \int_{s_0}^{t}\one_{ [s,t] }(\sigma) |P_{s,\sigma }\psi (\sigma, \cdot)(x) - P_{s_0, \sigma }\psi (\sigma, \cdot)(x_0)|\,d\sigma
\\
\\  & &\ds + \int_{s_0}^s  |P_{s_0,\sigma }\psi (\sigma, \cdot)(x)|\,d\sigma. \end{array}$$
Since for every $\sigma \geq s_0$ the function \[(s,x) \in\,[0,t]\times X\mapsto \one_{ [s,t] }(\sigma) P_{s, \sigma}\psi (\sigma , \cdot)(x)\] is continuous, and 
\[|P_{s,\sigma }\psi (\sigma, \cdot)(x) - P_{s_0, \sigma }\psi (\sigma, \cdot)(x_0)|\leq 2\|\psi \|_{\infty}\,\] by the Dominated Convergence Theorem the first integral vanishes as $s \downarrow s_0$, and $x\to x_0$. The second integral is bounded by $(s-s_0)\|\psi \|_{\infty}$, so that it vanishes too, as $s\downarrow s_0$, and $x\to x_0$. If $s<s_0$ we split 
\[u_1(s,x) - u_1(s_0, x_0)= 
\int_{s_0}^{  t} (P_{s,\sigma }\psi (\sigma, \cdot)(x) - P_{s_0, \sigma }\psi (\sigma, \cdot)(x_0))\,ds + \int_{s}^{s_0}  P_{s, \sigma}\psi (\sigma, \cdot)(x)\,d\sigma,\] and we argue in the same way. So, $u_1$ is continuous. Estimate  \eqref{stimasup_par} is immediate. 

Concerning statements (i) and (ii), the proof of the fact that $u_1(s, \cdot)\in C^n_b(X)$, for every $s\in [0, t]$, and for every $k\in \{1, \ldots, n\},$
\begin{equation}
\label{derivatak-esima}
\frac{\partial^ku_1}{\partial h_1 \ldots  \partial h_k}(t,  x)= \int_{0}^t D^kP_{s,\sigma}\psi (\sigma, \cdot)(x)(h_1, \ldots, h_k) \,d\sigma , \quad  s\in [0, t], \;x\in X, 
\end{equation}
is a consequence of Corollary \ref{Cor:stimeP_st}. Actually, by taking into account estimates \eqref{nth-derivatives} and \eqref{nth-derivativesHolder} respectively,  estimates \eqref{C^n0_par} and 
\eqref{C^nalpha_par} follow as well. 
The proof that the mapping
\[(s,x) \in\,[0, t)\times X\mapsto D^k u_1(s, \cdot)(x)(h_1, \ldots, h_k)\in\,\mathbb{R},\]  is continuous, for every $ k\in \{1, \ldots, n\}$, and $h_1, \ldots h_k\in X$,  is similar to the proof of the continuity of $u_1$. Indeed, 
for $s_0 < s \leq t$ and  $x$, $x_0\in X$ we split \[\frac{\partial^ku_1}{\partial h_1 \ldots\partial h_k}(s,  x) - \frac{\partial^ku_1}{\partial h_1 \ldots \partial h_k}(s_0,  x_0)= I_1 + I_2,\] where 
$$I_1 =  \int_{s_0}^{t}\one_{[s,t]} (\sigma) (D^k P_{s,\sigma}\psi (\sigma, \cdot)(x)- D^kP_{s_0,\sigma}\psi(\sigma, \cdot)(x_0))(h_1, \ldots, h_k) \,d\sigma , $$
and
$$I_2 =\int_{s_0}^t D^k P_{s, \sigma}\psi (\sigma, \cdot)(x)(h_1, \ldots, h_k) \,d\sigma,  $$
and from now on we proceed as before.  
\end{proof}

Now we have all the tools to prove our Schauder type theorems, arguing as in the paper \cite{LR} that deals with  autonomous problems. More precisely, it is sufficient to rewrite in our situation the proofs of Theorems 3.12 and 3.13 of \cite{LR} in  the strong-Feller case ($H=X$ according to the notation of \cite{LR}). 
For this reason we will write full details of the proofs only for some values of the parameters $\alpha$ and $\theta$.

 \begin{Theorem}
 \label{Th:Schauder_par}
Under Hypotheses \ref{H1}, \ref{H2}, and \ref{Hyp:Lambda} hold, let  $\varphi \in C_b(X)$, and  $\psi \in C_b([0,t]\times X)$ and let $u$ be defined by \eqref{u}. The following statements hold. 
  \begin{itemize}
 \item[(i)] If $1/\theta \notin \N$ and if $\varphi \in C^{1/\theta}_b(X)$, and $\psi \in C_b([0,t]\times X)$,  then $u\in C^{0, 1/\theta}_b ([0,t]\times X)$. Moreover, there exists $C= C(T)>0$, independent of $\varphi $ and $\psi $, such that 
 \begin{equation}
 \label{maggSchauder_par0}
 \|u\|_{ C^{0, 1/\theta}_b([0,t]\times X)} \leq C ( \|\varphi \|_{C^{1/\theta}_b(X)} + \|\psi \|_{\infty}). 
 \end{equation}
\item[(ii)] If $\alpha\in (0,1)$ and  $\alpha + 1/\theta \notin \N$, and if $\varphi \in C^{\alpha + 1/\theta}_b(X)$ and $\psi \in C^{0, \alpha}_b ([0,t]\times X)$,
   then $u\in C^{0,\alpha + 1/\theta}_b ([0,t]\times X)$. Moreover, there exists $C= C(T, \alpha)>0$, independent of $\varphi $ and $\psi $, such that 
 \begin{equation}
 \label{maggSchauder_par}
 \|u\|_{ C^{0,\alpha + 1/\theta}_b ([0,t]\times X)} \leq C  ( \|\varphi \|_{ C^{\alpha + 1/\theta}_b(X)} + \|\psi \|_{C^{0,\alpha  }_b ([0,t]\times X)} ). 
 \end{equation}
\end{itemize}
\end{Theorem}
 \begin{proof}
It  is sufficient to prove that the statements hold in the case $\varphi \equiv 0$, namely in the case when $u$ coincides with the function  $ u_1$ defined in \eqref{v0}. Indeed, thanks to Corollary \ref{reg_omogenea}, for any non integer $\beta >0$ the function $(s, x)\mapsto P_{s,t}\varphi (x)$ belongs to $C_b^{0, \beta} ([0,t]\times X)$, if $\varphi\in C_b^{\beta}(X)$. 

Taking into account Proposition 
 \ref{Pr:nonoptimal_par}, we have to check that for every $s\in [0,t]$, $u_1(s, \cdot)\in C^{\alpha +1/\theta}_b(X)$  with $\alpha =0$ in the case of  statement (i) and $\alpha \in (0, 1)$ in the case of statement (ii), with H\"older norm bounded by a constant independent of $s$. 

Let  $n\in \N\cup \{0\}$ be the integer part of $\alpha +1/\theta$. Here  we treat only the case   $n>0$, leaving the (similar, and even easier) case $n=0$ to the reader. 

By Proposition \ref{Pr:nonoptimal_par}  we already know that $u_1\in C^{0,n}_b([0, t]\times X)$. 
So, we have to prove  that $D^n u_1(s, \cdot)$ is $\alpha + 1/\theta -n$ -H\"older continuous with values in 
$\mathcal L^n(X)$, with H\"older constant independent of $s\in [0,t]$. 
We split  every partial derivative  $D^n u_1(s, y)(h_1, \ldots, h_n)$ as $a_h(s, y) + b_h(s, y)$, where   
\begin{equation}
\label{a_par}
a_h(s, y) : =  - \int_s^{(s+ \|h\|^{1/\theta})\wedge t}   D^n P_{s, \sigma}\psi(\sigma, \cdot)(y)(h_1, \ldots, h_n) \,d\sigma, \quad s\in [0,t], \;y\in X, 
\end{equation}
and 
\begin{equation}
\label{b_par}
b_h(s, y) = - \int_{(s+ \|h\|^{1/\theta}) \wedge t }^{t}  D^n P_{s, \sigma}\psi(\sigma, \cdot) (y)(h_1, \ldots, h_n) \,d\sigma , \quad \quad s\in [0,t], \;y\in X. 
\end{equation}

Let  us consider statement (i), where $\alpha =0$. In  this case,  $\psi \in C_b([s,t]\times X)$, $n\theta \in (0, 1)$, $(n+1)\theta >1$. 
Estimate \eqref{nth-derivatives} yields  
$$\begin{array}{lll}
| a_h(s, x+h) - a_h(s, x) | & \leq & | a_h(s,  x+h)| + |a_h(s, x) | \leq  \ds 2 
K_{n }    \int_s^{(s+ \|h\|^{1/\theta})\wedge t} \sigma ^{-n\theta} d\sigma \|\psi\|_{\infty} \prod_{j=1}^n\|h_j\| 
\\
\\
& \leq  & \ds
\frac{2K_{n } }{ 1-n\theta} \|h\|^{(1-n\theta)/\theta}  \|\psi\|_{\infty}\prod_{j=1}^n\|h_j\| .
\end{array}$$

If $\|h\|^{1/\theta}\geq (t-s)$, $b_h(s, \cdot)$ vanishes. To estimate $| b_h(s, x+h) - b_h(s, x) | $ when $\|h\|^{1/\theta}<t-s$ we use again  \eqref{nth-derivatives} and for every $\sigma \in (s,t)$ we get
\begin{equation}
\label{eq:3}
\begin{array}{lll}
\|D^n P_{s,\sigma} \psi(\sigma, \cdot)(x+h) - D^n P_{s,\sigma} \psi(\sigma, \cdot)(x)\|_{\mathcal L^n(X)}&  \leq & \sup_{y\in X} \|D^{n+1}P_{s,\sigma} \psi(\sigma, \cdot)(y)\|_{\mathcal L^n(X)} \|h\|
\\
\\ &
\leq & \ds \frac{ K_{n+1 } }{  (\sigma -s)^{(n+1 )\theta}   } \|\psi\|_{\infty} \|h\| , 
\end{array}
\end{equation}
so that 
$$\begin{array}{lll}
| b_h(s,  x+h) - b_h(s, x) | & \leq &    \ds   \int_{s+ \|h\|^{1/\theta} }^{t}  \frac{K_{n+1 }}{(\sigma -s)^{(n+1 )\theta}} d\sigma \, \|\psi \|_{\infty}\|h\|
 \prod_{j=1}^n\|h_j\| 
\\
\\
& \leq  &\ds  \frac{K_{n+1 }}{ (n+1 )\theta -1} \|h\|^{ 1  /    \theta - n  } \| \psi \|_{\infty} \prod_{j=1}^n\|h_j\|  .
\end{array}$$
Summing up we get
$$| (D^n u_1(s, x +h) -  D^n u_1(s, x))(h_1, \ldots, h_n)| \leq C \|h\|^{1/\theta -n   } \| \psi \|_{\infty}\prod_{j=1}^n\|h_j\| , \quad 0\leq s\leq t\leq T,  $$
with 
$$C =      \frac{2K_{n } }{ 1-n\theta} +  \frac{K_{n+1 }}{ (n+1 )\theta -1} . $$
Therefore,  
$[D^n u_1(s, \cdot)]_{C^{1/\theta -n  }(X; \mathcal L^n(X))} \leq C \|\psi \|_{\infty}$ for every $s\in [0, t]$. This estimate and \eqref{stimasup_par}   give \eqref{maggSchauder_par0} for $n\geq 1$.

The proof of statement (ii) is similar. Now we have  $\psi \in C^{0, \alpha}([s,t]\times X)$, with  $\alpha \in (0, 1)$, $(n-\alpha)\theta \in (0, 1)$, and  $(n+1-\alpha)\theta >1$.  
Estimate \eqref{nth-derivativesHolder} yields
$$\begin{array}{l}
| a_h(s, x+h) - a_h(s, x) |   \leq    \ds 2 
K_{n, \alpha}  \int_s^{(s+ \|h\|^{1/\theta})\wedge t}    \frac{1}{(\sigma -s)^{(n-\alpha)\theta}}\|\psi (\sigma, \cdot)\Vert_{C^\alpha_b(X)}d\sigma \prod_{j=1}^n\|h_j\|
\\
\\
  =   \ds    \frac{2K_{n, \alpha} }{ 1-(n-\alpha)\theta} \|h\|^{(1-(n-\alpha)\theta)/\theta} \prod_{j=1}^n\|h_j\| \sup_{s\leq r \leq t}\|\psi (r, \cdot)\|_{C^\alpha_b(X)}. 
\end{array}$$
Again, if $\|h\|^{1/\theta}\geq (t-s)$, $b_h(s, \cdot)$ vanishes. To estimate $| b_h(s, x+h) - b_h(s, x) | $ when $\|h\|^{1/\theta}<t-s$ we use now  \eqref{nth-derivativesHolder} that gives, for every $\sigma \in (s,t)$, 
\begin{equation}
\label{eq:4}
\begin{array}{lll}
\|D^n P_{s,\sigma} \psi(\sigma, \cdot)(x+h) - D^n P_{s,\sigma} \psi(\sigma, \cdot)(x)\|_{\mathcal L^n(X)} & \leq & \sup_{y\in X} \|D^{n+1}P_{s,\sigma} \psi(\sigma, \cdot)(y)\|_{\mathcal L^n(X)} \|h\|
\\
\\
&
\leq & \ds \frac{ K_{n+1 ,\alpha } }{  (\sigma -s)^{(n+1 -\alpha)\theta}   } \|\psi\|_{\infty} \|h\| .
\end{array}
\end{equation}
Therefore, 
$$\begin{array}{l}
| b_h(s, x+h) - b_h(s,  x) |   \leq     \ds  \int_{s+ \|h\|^{1/\theta} }^{t}   \frac{K_{n+1, \alpha } }{(\sigma -s)^{(n+1-\alpha)\theta}} \|\psi(\sigma,  \cdot)\|_{C^\alpha_b(X)}d\sigma \,\|h\|
 \prod_{j=1}^n\|h_j\|
\\
\\
  \leq    \ds \frac{K_{n+1 }}{ (n+1 -\alpha)\theta -1} \|h\|^{ 1/ \theta - n  +\alpha } \prod_{j=1}^n\|h_j\|  \sup_{s\leq r\leq t}\|\psi (r, \cdot)]_{C^\alpha_b(X)}. 
\end{array}$$
Summing up we get
$$| (D^n u_1(s, x +h) -  D^nu_1(s, x))(h_1, \ldots, h_n)| \leq C_1 \|h\|^{1/\theta -n +\alpha } \prod_{j=1}^n\|h_j\| \sup_{s\leq r\leq t}\|\psi (r, \cdot)\|_{C^\alpha_b(X)}$$
with 
$$C_1 =   \frac{2K_{n, \alpha} }{ 1-(n-\alpha)\theta} +  \frac{K_{n+1, \alpha}}{ (n+1-\alpha)\theta -1}     . $$
Therefore,  $[D^nu_1(s, \cdot)]_{C^{1/\theta -n  +\alpha}_b(X; \mathcal L^n(X))} \leq C_1 \|\psi\|_{C^{0,\alpha}_b([0,t]\times X)}$  for every $s\in [0,t]$. This estimate and \eqref{stimasup_par}   give \eqref{maggSchauder_par} in the case $n\geq 1$. 
  \end{proof}

\begin{Theorem}
\label{Zygmund_par}
Under Hypotheses  \ref{H1}, \ref{H2}, and \ref{Hyp:Lambda} hold, let  $\varphi \in C_b(X)$, and $\psi \in C_b([0,t]\times X)$ and let $u$ be defined by \eqref{u}. The following statements hold. 
  \begin{itemize}
 \item[(i)] If $1/\theta = k\in \N$ and $\varphi \in Z^{k}_b(X)$, then $u \in Z^{0,k}_b([0,t]\times X)$ and there exists $C= C(T)>0$, independent of $\varphi $ and $\psi $, such that 
 \begin{equation}
 \label{maggZygmund_par0}
 \|v\|_{ Z^{0,k}_b([0,t]\times X)} \leq C ( \|\varphi \|_{Z^{k}_b(X)} + \|\psi \|_{\infty}). 
 \end{equation}
  \item[(ii)] 
If  $\alpha\in (0,1)$, and  $\alpha + 1/\theta = k \in \N$, and if  $\varphi \in Z^{k}_b(X)$ and $\psi \in  C^{0,\alpha}_b([0,t]\times X)$, then $u \in Z^{0,k}_b([0,t]\times X)$, and there exists $C= C(T, \alpha)>0$, independent of $\varphi $ and $\psi$, such that 
 \begin{equation}
 \label{maggZygmund_par}
 \|u\|_{ Z^{0,k}_b([0,t]\times X)}   \leq C  ( \|f\|_{Z^{k}_b(X)} + \|\psi\|_{C^{0,\alpha}_b([0,t]\times X)} ). 
 \end{equation}
\end{itemize}
 \end{Theorem}
 \begin{proof}
 We know by Corollary \ref{reg_omogenea}  that for every $\varphi \in Z^{k}_H(X)$ the function $(s,x) \mapsto P_{s,t}\varphi (x)$ belongs to $Z^{0,k}([0,t]\times X)$, and estimate \eqref{Zygmund-Zygmund} holds. So it is enough to prove that the statements hold for $\varphi \equiv 0$, in which case $u$ coincides with the function $u_1$ defined by \eqref{v0}.

In fact we give a proof only in the case $k>1$, leaving the (easier) case $k=\theta=1$ to the reader. 

We recall that $k= 1/\theta$ in statement (i) and $k= \alpha + 1/\theta$ in statement (ii). In both cases, Proposition \ref{Pr:nonoptimal_par} yields  $u_1\in C^{0, k-1}([0,t]\times X)$. 
We have to prove   that $[D^{k-1}u_1(t, \cdot)]_{Z^1(X, \mathcal L^{k-1}(X))}$ is bounded by a constant independent of $t$. To this aim, fixed any $h, h_1, \ldots, h_{k-1}\in X$, for every $s\in [0,t]$ and $y\in X$  we split 
 $D^{k -1}u_1(s, y)(h_1, \ldots, h_{k-1})$ as $a_h(s,  y) + b_h(s, y)$, where $a_h$ is defined in \eqref{a_par} and $b_h$ is defined in \eqref{b_par}, both with $n=k-1$. 

So, we have 
 $$\begin{array}{l}
| a_h(s, x+2h) - 2 a_h(s,  x+h) + a_h(s, x)|
\\
\\
\ds \leq   \int_s^{(s+ \|h\|^{1/\theta})\wedge t}     |(D^{k-1} P_{s,\sigma}\psi(\sigma, \cdot)(x+2h) - 2D^{k-1}P_{s,\sigma}\psi(\sigma, \cdot)(x+h) 
\\
\\
\hspace{22mm} + D^{k-1}P_{s,\sigma}\psi(\sigma, \cdot) (x ))(h_1, \ldots, h_{k-1})|\,d\sigma , \end{array}$$
which  is bounded by  
$$  \int_s^{(s+ \|h\|^{1/\theta})}  \frac{ 4 K_{k-1}}{(\sigma -s)^{(k-1)\theta}}d\sigma \|\psi\|_{\infty} \prod_{j=1}^{k-1}\|h_j\|  \leq 
4 kK_{k-1}  \|\psi\|_{\infty} \|h\|  \prod_{j=1}^{k-1}\|h_j\|,$$ 
if $k = 1/\theta$, and by 
$$\begin{array}{l}
\ds \int_s^{(s+ \|h\|^{1/\theta})}  \frac{4K_{k-1, \alpha} }{(\sigma - s)^{(k-1 -\alpha)\theta}}\| \psi(\sigma, \cdot)\|_{C^{\alpha}_b(X)}d\sigma  \prod_{j=1}^{k-1}\|h_j\|
\\
\\
\ds \leq 4(k-\alpha) K_{k-1, \alpha}
 \|h\| \prod_{j=1}^{k-1}\|h_j\| \sup_{0\leq r\leq t}\| \psi(\sigma, \cdot)\|_{C^{\alpha}_b(X)} , \end{array}$$
if  $\psi \in C^{0, \alpha}_b([0,t]\times X)$ with $\alpha \in (0, 1)$ and 
$k= \alpha + 1/\theta$. 
If $\|h\|^{1/\theta}\geq (t-s)$, $b_h(s, \cdot) =0$. If $\|h\|^{1/\theta} < (t-s)$,  we have 
\begin{equation}
\label{bZygmund}
\begin{array}{l}
| b_h(s,  x+2h) - 2 b_h(s, x+h) + b_h(s, x)|  \leq
\\
\\
\ds   \int_{s+ \|h\|^{1/\theta} }^{t}  |(D^{k-1}P_{s, \sigma} \psi(\sigma , \cdot)(x+2h) - 2D^{k-1}P_{s, \sigma} \psi(\sigma , \cdot)(x+h)+ D^{k-1}P_{s, \sigma} \psi(\sigma , \cdot)(x))(h_1, \ldots, h_{k-1})|    \,d\sigma . 
\end{array}
\end{equation}
To estimate the right-hand side we recall that for every $f\in C^2_b(X)$  we have
\begin{equation}
\label{c^2}
|f(x+2h) - 2f(x+h) + f(x)| \leq \sup_{y\in X} \|D^2f(y)\|_{\mathcal L^2(X)} \|h\|^2, \quad x, \;h\in X. 
\end{equation}
Applying \eqref{c^2} to the function $f(y):= D^{k-1}P_{s, \sigma} \psi(\sigma , \cdot)(y)(h_1, \ldots, h_{k-1})$, we see that 
the right-hand side of \eqref{bZygmund} is bounded by 
$$   \int_{s+ \|h\|^{1/\theta} }^{t} \frac{  K_{k+1}}{(\sigma -s)^{(k+1)\theta}} \|\psi \|_{\infty} d\sigma \, \|h\|^2    \prod_{j=1}^{k-1}\|h_j\|\leq    k K_{k+1}  \|h\| \|\psi \|_{\infty}  \prod_{j=1}^{k-1}\|h_j\|, $$
if $k = 1/\theta$, and by 
$$\begin{array}{l} \ds  \int_{s+ \|h\|^{1/\theta} }^{t}
\frac{K_{k+1, \alpha}}{(\sigma -s)^{(k+1-\alpha)\theta}} \|\psi (\sigma, \cdot)\Vert_{C^{\alpha}_b(X)}d \sigma \, \|h\|^2 \prod_{j=1}^{k-1}\|h_j\|
\\
\\
\ds  \leq  (k-\alpha) K_{k+1, \alpha}
 \|h\| \prod_{j=1}^{k-1}\|h_j\|_H  \sup_{0\leq r\leq t}\|\psi (r, \cdot)\Vert_{C^{\alpha}_b(X)}, \end{array}$$
if  $\psi \in C^{0,\alpha}_b([0,t]\times X)$, with $\alpha \in (0, 1)$ and 
$k= \alpha + 1/\theta$. 
Summing up, we estimate 
\[|[D^{k-1}u_1(s, \cdot)(x+2h) -2D^{k-1}u_1(s, \cdot)(x+h) + D^{k-1}u_1(s, \cdot)(x)] (h_1, \ldots h_{k-1})|\] 
 by 
 $$ k(4  K_{k-1} + K_{k+1})\prod_{j=1}^{k-1}\|h_j\| \|\psi \|_{\infty} \|h\|, $$
 if $1/\theta = k$, and by  
 $$ (k-\alpha)(4  K_{k-1, \alpha} + K_{k+1, \alpha})\prod_{j=1}^{k-1}\|h_j\| \sup_{0\leq r\leq t}\|\psi (r, \cdot)\|_{C^{\alpha}_b(X)} \|h\|, $$
 if  $\psi \in C^{0,\alpha}_b([0,t]\times X)$  with $\alpha \in (0, 1)$ and $ \alpha + 1/\theta = k$. This implies that 
 $u_1(t, \cdot) \in Z^{k}(X)$, with Zygmund seminorm bounded by $ k(4  K_{k-1} + K_{k+1})  \|\psi \|_{\infty}  $ in the first case,  and by  $(k-\alpha)(4  K_{k-1, \alpha} + K_{k+1, \alpha})\|\psi _{C^{\alpha}_b(X)} \|h\|$, in the second case. Such estimates and \eqref{Zygmund-Zygmund}   yield \eqref{maggZygmund_par0} and \eqref{maggZygmund_par}, respectively. 
 \end{proof}

\section{Examples}
\label{Sect:Examples}

\subsection{Example 1. }
As a first example we consider self-adjoint operators $A(t)$, $B(t)$, in diagonal form with respect to the same Hilbert basis $\{e_k:\; k\in \N\}$, namely
$$A(t)e_k = \alpha_k(t)e_k, \quad B(t)e_k = \beta_k(t)e_k, \quad 0\leq t\leq T, \ \ \ \  k\in \N,  $$
with continuous coefficients $\alpha_k$, $\beta_k$. We set 
$$\mu_k :=\min_{t} \alpha_k(t), \ \ \ \  \lambda_k:= \max_t \alpha_k(t),$$
and we assume that the sequence $(\lambda_k)$ is bounded from above, 
\begin{equation}
\label{suplambdak}
\sup_{k\in \N} \lambda_k := \lambda_0 <+\infty. 
\end{equation}
Then the family $\{A(t):\; 0\leq t\leq T\}$  generates a strongly continuous evolution operator $U(t,s)$,  given by 
$$U(t,s)e_k = \exp \left( \int_s^t \alpha_k(\tau)d\tau\right) e_k, \quad 0\leq s\leq t\leq T, \ \ \ \  k\in \N,   $$
so that Hypothesis \ref{H1}(i) is satisfied. 

Moreover, we assume that the coefficients $\beta_k$ are uniformly bonded, namely there is $M>0$ such that 
\begin{equation}
\label{betak}
|\beta_k(t)|\leq M, \quad 0\leq t\leq T, \ \ \ \  k\in \N. 
\end{equation}
Then the operators $B(t)$ belong to $\mathcal L(X)$, and $\sup_{0\leq t\leq T}\|B(t)\|_{\mathcal L(X)} \leq M$. So, Hypothesis \ref{H1}(ii)  is satisfied too. 

The operators $Q(t,s)$ are given by  
$$Q(t,s) e_k = \int_s^t \exp \left(2 \int_\sigma^t \alpha_k(\tau)d\tau\right) (\beta_k(\sigma))^2d\sigma e_k:= t_k(t,s) e_k, \quad 0\leq s\leq t\leq T, \; k\in \N. $$
Therefore, Hypothesis \ref{H1}(iv) is satisfied provided
\begin{equation}
\label{traccia}
\sum_{k=1}^\infty t_k(t,s) <+\infty, \quad 0\leq s<t\leq T. 
\end{equation}
An obvious sufficient condition for \eqref{traccia} to hold is, for instance, that $\lambda_k $ is eventually negative, and 
\begin{equation}
\label{suff_per_traccia}
 \sum_{k =k_0}^{+ \infty} \frac{\|\beta_k\|_{\infty}^2}{|\lambda_k|}  < +\infty ,
 \end{equation}
for some $k_0\in \N$. 

Since $Q(t,s)^{1/2}e_k = (t_k(t,s))^{1/2} e_k$ for every $t$, $s$, $k$, Hypothesis \ref{H2} is satisfied if $t_k(t,s)$ is eventually positive for every $t$, $s$ (say, for $k\geq k_0$), and 

\begin{equation}
\label{strongfeller}
\sup_{k\geq k_0} \exp\left(2\int_s^t \alpha_k(\tau)d\tau\right) (t_k(t,s))^{-1} < +\infty, \quad 0\leq s<t\leq T. 
\end{equation}
In this case the above supremum is equal to the square of the norm of the operator $\Lambda (t,s)$ defined in \eqref{Lambda(t,s)}. 
 
A sufficient condition for \eqref{strongfeller} to hold is that $\lambda_k$ is eventually negative, $b_k:= \min_{0\leq t\leq T} |\beta_k(t)|$ is eventually $ \neq 0$ and
$$\sup_{k\in \N: \,b_k\neq 0} \frac{|\mu_k|}{b_k^2 (e^{-2\lambda_k(t-s)} - e^{-2(\lambda_k-\mu_k)(t-s)})}     <+\infty . $$
Hypothesis \ref{Hyp:Lambda} holds provided the constants $\mu_k$, $\lambda_k$, $\beta_k$ have suitable powerlike behavior. For instance, if 
 
$ - c_1k^\alpha \leq \mu_k\leq \lambda_k \leq   - c_2 k^\alpha$ with $c_1\geq c_2 >0$, $\alpha >0$,   $b_k \geq c_3 k^{-\beta}$, with $\beta \geq 0$, $c_3>0$, 
we obtain $e^{-2\lambda_k(t-s)} - e^{-2(\lambda_k-\mu_k)(t-s)} \geq e^{2c_2k^\alpha (t-s)} - 1$ 
and therefore
$$\sup_{k\in \N} \frac{|\mu_k|}{b_k^2 (e^{-2\lambda_k(t-s)} - e^{-2(\lambda_k-\mu_k)(t-s)})}   
 \leq \sup_{k\in \N} \frac{c_2k^{\alpha + 2\beta} }{c_3^2 (e^{2c_2 k^\alpha (t-s)} - 1)} \leq C(t-s)^{-(\alpha + 2\beta)/\alpha}$$
with 
$$C =  \frac{c_2}{c_3^2 (2c_2)^{(\alpha + 2\beta)/\alpha} }\sup_{y>0} \frac{y^{(\alpha + 2\beta)/\alpha} }{e^y -1}, $$
so that \ref{Hyp:Lambda} holds with $\theta = 1/2 + \beta/\alpha$. 

%

Let us prove that the exponent $ 1/2 + \beta/\alpha$ is optimal, under the additional assumption
 $\|b_k\|_{\infty} \leq c_4k^{-\beta}$ for some $c_4>0$. In this case we also have
$$\begin{array}{lll}
\|\Lambda (t,s)\|_{\mathcal L(X)}^2 &=&\ds 
\sup_{k\in \N} \exp\left(2\int_s^t \alpha_k(\tau)d\tau\right) (t_k(t,s))^{-1} \geq 
\sup_{k\in \N}  \frac{2c_1k^{\alpha + 2\beta} e^{-2c_2k^\alpha (t-s)}}{c_4^2 (1-e^{-2c_1 k^\alpha (t-s)} )} 
\\
\\
&=& \ds \frac{2c_1}{c_4^2 (t-s)^{1+ 2\beta/\alpha} }  \sup_{k\in \N} \varphi (k^\alpha (t-s)), 
\end{array}$$
where
$$\varphi(y) = \frac{y^{2\beta/\alpha +1}e^{-2c_2y}}{1- e^{-2c_1y}}, \quad y>0. $$
Fix any   $[a,b]\subset (0, +\infty)$ and set  $M_0:= \min\{ \varphi(y):\; y\in [a,b]\}$, so that $M_0>0$. Let $t-s$ be small enough, in such a way that 
\[[a/(t-s), b/(t-s)] \cap \{k^\alpha: \; k\in \N\}\neq \emptyset\]  (it is sufficient that $t-s < (b^{1/\alpha} -a^{1/\alpha})^\alpha$), and choose any $k_0\in \N$ such that $k_0^\alpha$ belongs to such intersection. So, we have $k_0^\alpha (t-s)\in [a,b]$, and therefore
$$\sup_{k\in \N}    \varphi (k^\alpha (t-s)) \geq   \varphi (k_0^\alpha (t-s)) \geq M_0,  $$
which yields
$$\|\Lambda (t,s)\|_{\mathcal L(X)} \geq \frac{(2c_1 M_0)^{1/2}}{c_4} (t-s)^{-1/2 -\beta/\alpha}$$
for $t-s$ small.

\subsection{Example 2. }
As a second example, we consider the evolution operator $U(t,s)$ in $X=L^2( \mathcal O)$ associated to an evolution equation of parabolic type,  
\begin{equation}
\label{eqpar}
\left\{
\begin{array}{l}
u_t(t,x)  = \mathcal A(t) u(t, \cdot)(x), \quad ( t,x)\in (s,T)\times \mathcal O, 
\\
\\
u(t,x) = 0, \quad (t,x)\in  (s,T)\times  \partial \mathcal O, 
\\
\\
u(s, x) = u_0(x), \quad x\in  \mathcal O. 
\end{array}\right. 
\end{equation}
Here $0\leq s<T$ and $\mathcal O$ is a bounded open set in $\R^d$ with smooth enough ($C^2$) boundary. The differential operators $\mathcal A (t)$ are defined by 
\begin{equation}
\label{A(t)}
\mathcal A (t)\varphi(x) = \sum_{i,j=1}^d a_{ij}(t,\cdot )D_{ij}\varphi (x) + \sum_{i=1}^d a_i(t, x)D_i\varphi(x) + a_0(t,x) \varphi (x), \quad t\in [0,T], \; x\in \mathcal{O}, 
\end{equation}
and we make the following assumptions on the coefficients.

\begin{Hypothesis}
\label{H4}
$a_{ij}\in C^{0,1+\rho}([0,T]\times \overline{\mathcal O})$, $a_i\in C^{0, \rho}([0,T]\times \overline{\mathcal O})$, for some $\rho>0$. There exists $\nu >0$ such that for all $\xi\in \R^d$
\begin{equation}
\label{ellitticitˆ}
\sum_{i,j =1}^d a_{ij}(t,x)\xi_i\xi_j \geq \nu|\xi|^2, \quad t\in [0,T], \; x\in \mathcal O. 
\end{equation}
\end{Hypothesis}

The assumptions on the operators $B(t)$, $0\leq t\leq T$, are the following. 

\begin{Hypothesis}
\label{H5}
There  exists $q\geq 2$, $q > d$,  such that for a.e. $t\in [0,T]$, $B(t)\in \mathcal L(L^{2}(\mathcal O), L^q(\mathcal O))$ has bounded inverse, 
and 
\[\displaystyle{\esssup _{0<t<T} (\|B(t)\|_{ \mathcal L(L^{2}(\mathcal O), L^q(\mathcal O))} + \|B(t)^{-1}\|_{ \mathcal L(L^{q}(\mathcal O), L^2(\mathcal O))} )
<+\infty.}\]
 Moreover for every $\varphi \in L^2(\mathcal O)$ the mapping $[0,T] \mapsto L^q(\mathcal O) $,  $t\mapsto B(t)\varphi$ is measurable. 
\end{Hypothesis}


Notice that, since $q\geq 2$, $B(t)\in \mathcal L(X)$ for a.e. $t\in (0,T)$, and Hypothesis \ref{H1}(ii) on $B$ is satisfied. We use the notation $B^\star(t)$ to denote both dual operators of $B(t): L^2(\mathcal O)\mapsto L^2(\mathcal O)$ and  $B(t): L^2(\mathcal O)\mapsto L^q(\mathcal O)$, since the latter dual operator is just a restriction of the former.

Using the classical theory of quadratic forms in Hilbert spaces (e.g. \cite{DL}), it is possible to prove that for every $u_0\in X$,  \eqref{eqpar} has a unique weak solution, 
namely there exists a unique $u\in W:= L^2((s,T); H^{1}_0(\mathcal O)) \cap W^{1,2}((s,T); H^{-1}(\mathcal O)) $ such that 
$$-\int_s^T \langle u(t),v'(t)\rangle \,dt + \int_s^T {\bf a}(t,u(t),v(t))\,dt = \langle u_0, v(0)\rangle$$
  for every $v\in W$ satisfying $v(T)=0$. Here  $\langle \cdot, \cdot\rangle$ is the scalar product in $L^2(\mathcal O)$ and ${\bf a}(t, \cdot , \cdot )$ is the quadratic form associated to the operator $\mathcal A(t)$ in $H^1_0(\mathcal O)$, namely
$$\begin{array}{lll}
{\bf a}(t, \varphi, \psi)  & = & \ds \int_{\mathcal O} \sum_{i,j=1}^d a_{ij}(t,x)D_j\varphi(x)D_i\psi(x)\,dx 
\\
\\
&& \ds - \int_{\mathcal O} \left(\sum_{j=1}^d (a_j(t,x) + \sum_{i=1}^d D_ia_{ij}(t,x))D_j\varphi (x) + a_0(t,x) \varphi(x)\right) \psi(x) \,dx 
\end{array}$$
for every $t\in [0,T]$, and $\varphi$, $\psi\in H^1_0(\mathcal O)$. See e.g. \cite[Thm. 2.4]{D}. 

Setting $U(t,s)u_0:= u(t)$, where $u$ is the unique weak solution to \eqref{eqpar}, $U(t,s)$ turns out to be an evolution operator in $X$. 

In the same paper \cite{D}  it is shown that $U(t,s)$ may be extended to the whole $L^1(\mathcal O)$, and the extension (still called $U(t,s)$) belongs to $ \mathcal L(L^1(\mathcal O), L^\infty (\mathcal O))$, therefore it is represented by
\begin{equation}
\label{reprU(t,s)}
U(t,s)\varphi (x) = \int_{\mathcal O} k(x,y,t,s)\varphi(y)\,dy,\quad  \varphi \in L^1(\mathcal O), 
\end{equation}
where the kernel $k(\cdot, \cdot, t,s)$ belongs to $L^\infty( \mathcal O \times \mathcal O)$ for every $0\leq s<t\leq T$. Moreover, $k$ satisfies gaussian bounds, namely there are $M$, $m>0$ such that 
\begin{equation}
\label{gaussian}
|k(x,y,t,s)| \leq \frac{M }{(t-s)^{d/2}} \exp \left(- \frac{|x-y|^2}{m(t-s)} \right), \quad x, \;y\in \mathcal O, \; 0\leq s<t\leq T. 
\end{equation}
See also \cite{A}. Together with Hypothesis \ref{H5}, such an estimate allows to prove that Hypothesis \ref{H1}(iv) holds, as shown in the next lemma. 

\begin{Lemma}
For every $0\leq s<t\leq T$, the operator $Q(t,s)$ defined in \eqref{g,Q}  has finite trace. 
\end{Lemma}
\begin{proof}
Let $\{ e_k:\; k\in \N\}$ be any Hilbert basis of $L^2(\mathcal O)$. The trace of $Q(t,s)$ is given by
\[\int_s^t  \sum_{k=1}^{\infty}
\Vert B^\star(\sigma) U^\star (t,\sigma) e_k\Vert^{2}_{L^2(\mathcal{O})} d\sigma, \]
where, by  the representation formula \eqref{reprU(t,s)}, 
\[(U^\star (t,\sigma)e_k )(z) = \int_{\mathcal O} k(x,z,t,\sigma) e_k(x) \,dx,\ \ \ \ \mbox{ a.e.}\  z\in \mathcal O,\]
 and therefore 
 \[(B^\star(\sigma) U^\star (t,\sigma) e_k) (y) =  \int_{\mathcal O} (B^\star(\sigma) k(x,\cdot ,t,\sigma))(y) e_k(x) \,dx,\ \ \ \ \mbox{ a.e.}\  y\in \mathcal O.\]
  It follows
\begin{equation}
\label{stimatraccia}
\begin{array}{lll}
\text{Tr} \; Q(t,s) & = & \ds \int_s^t \int_{\mathcal O} \sum_{k=1}^{\infty} \left( \int_{\mathcal O} B^\star(\sigma) k(x,\cdot, t,\sigma)(y) e_k(x)\,dx\right)^2 \,dy\, d\sigma 
\\
\\
& = & \ds \int_s^t   \int_{\mathcal O} \left( \int_{\mathcal O} (B^\star (\sigma) k(x, \cdot, t,\sigma)(y))^2dx\right) dy \,d\sigma
\\
\\
& = & \ds \int_s^t   \int_{\mathcal O}\left(  \int_{\mathcal O} (B^\star (\sigma) k(x, \cdot, t,\sigma)(y))^2dy\right) dx\,d\sigma
\end{array}\end{equation}
Now, since $B(\sigma)\in \mathcal L (L^2(\mathcal O), L^q(\mathcal O))$ then $B^\star(\sigma ) \in   \mathcal L (L^{q'}(\mathcal O), L^2(\mathcal O))$, and  for every $x\in \mathcal O$ we have
$$\|B^\star (\sigma) k(x, \cdot, t,\sigma)\|_{L^2(\mathcal O)} \leq \|B^\star (\sigma) \|_{ \mathcal L (L^{q'}(\mathcal O), L^2(\mathcal O))} 
\|k(x, \cdot, t, \sigma )\|_{ L^{q'}(\mathcal O)} \leq C\|k(x, \cdot, t, \sigma)\|_{ L^{q'}(\mathcal O)} . $$

 By estimates \eqref{gaussian}, for every $x\in \mathcal O$  and $p>1$   we have
\begin{equation}
\label{stimak}
\|k(x, \cdot, t, \sigma)\|_{ L^{p}(\mathcal O)} ^{p} \leq M(t-\sigma)^{-dp/2} \int_{\mathcal \R^d} e^{-p|x-y|^2/m(t-\sigma)} dy =: M_p(t-\sigma)^{d(1-p)/2}, 
\end{equation}
so that, taking $p=q'$, 
%
%
%
 
$$\|B^\star (\sigma) k(x, \cdot, t,\sigma)\|_{L^2(\mathcal O)}^2  \leq C^2 M_{q'}^{2/q'} (t-\sigma)^{-d/q}$$
(independent of $x$). Replacing in \eqref{stimatraccia}, since we are assuming  $q>d$, it follows that the trace of $Q(t,s)$ is finite. 
\end{proof}

Hypothesis \ref{H4} allows to use the results of the paper \cite{PS}, see \cite[Sect. 4]{PS}. There, a strongly continuous  evolution operator $U_q(t,s)$ is constructed in the space
$X_q:= L^q(\mathcal O)$, for $q\in (1, +\infty)$, in such a way that, setting $D_q:= W^{2,q}(\mathcal O) \cap W^{1,q}_0(\mathcal O)$, for every $\varphi\in L^p((s,t); X_q)$ and $u_0\in (X_q, D_q)_{1-1/p, p}$  with $p\in (1, +\infty)$,  the function $u(t) :=U_q(t,s)u_0$ is  the unique strong solution to  \eqref{eqpar}, namely it is the unique function  \[u \in\,L^p((s,T); D_q) \cap W^{1,p} ((s,T); X_q) \cap C([s,t];  (X_q, D_q)_{1-1/p, p})\] that  satisfies 
\begin{equation}
\label{strong}
\left\{ \begin{array}{l}
u'(\tau) = A_q(\tau) u(\tau) + \varphi(\tau), \quad a.e. \; \tau \in (s,t),
\\
\\
u(s) = u_0
\end{array}\right.
\end{equation}
where \[A_q(\tau) : D_q\mapsto X_q,\ \ \  A_q(\tau)\varphi := A(\tau)\varphi\]  is the realization of $\mathcal A(\tau) $ in $X_q$. 
Taking $p=q=2$, we obtain that $U_2(t,s)$ coincides with our $U(t,s)$, since for every $u_0\in W^{1,2}_0(\mathcal O) = (X_2,D_2)_{1/2, 2}$ the function $u(t) =U_2(t,s)u_0$ is a strong solution to \eqref{eqpar}, so that it is a weak solution. By uniqueness of the weak solution, the bounded operators $U(t,s)$ and $U_2(t,s)$ coincide on a dense subset of $X$, and therefore they coincide on the whole $X$. Still by uniqueness, for $q>2$ the operators $U_q(t,s)$ are the parts of $U(t,s)$ in $X_q$. 
 Therefore, Hypotheses \ref{H1}(i)-(iv) about $U(t,s)$ are satisfied. 
 
 Let us  check that Hypotheses \ref{H2} and  \ref{Hyp:Lambda} are satisfied too.
By  Theorem 2.5 of  \cite{PS},  for every $0\leq s<t\leq T$ and $\varphi \in L^2((s,t);X_q)$, the problem 
\begin{equation}
\label{immagine}
\left\{ \begin{array}{l}
v'(\tau) = A_q(\tau) v(\tau) + \varphi(\tau), \quad s<\tau<t,
\\
\\
v(s) =0
\end{array}\right.
\end{equation}
has a unique solution $v\in L^2((s,t); D_q)\cap W^{1,2}((s,t);X_q)$, given by the variation of constants formula
$$v(\tau) = \int_s^\tau U(\tau, \sigma) \varphi(\sigma)\,d\sigma, \quad s\leq \tau\leq t, $$
 and there exists $C>0$ independent of $s$, $t$, $\varphi$ such that
$$\|v\|_{L^2((s,t);D_q)} + \|v\|_{W^{1,2}((s,t); X_q)} \leq C \|\varphi\|_{L^2((s,t);X_q)}. $$
 
 Therefore, the mapping 
$$\{ v\in L^2((s,t);D_q)\cap W^{1,2}((s,t);X_q):\; v(s)=0\}  \mapsto L^2((s,t);X_q), \quad v\mapsto \Phi(v):= v'-A_q(\cdot)v $$
is an isomorphism. 
We recall now that for every couple of Banach spaces $X$, $D$, such that $D\subset X$ with continuous embedding,  the space $L^2((s,t);D)\cap W^{1,2}((s,t);X)$ is continuously embedded in 
 $C([s,t]; (X,D)_{1/2, 2})$, and the range of the trace operator $v\mapsto Tv:= v(t)$ is precisely $(X,D)_{1/2, 2}$.  It follows that the range of the mapping 
$$L^2((s,t);X_q)\mapsto X_q, \quad \varphi\mapsto \int_s^t U(t,\sigma) \varphi(\sigma)d\sigma  =T \Phi^{-1}\varphi $$
is equal to $(X_q, D_q)_{1/2, 2}$. By Hypothesis \ref{H5}, the operator \[L^2((s,t);X_2) \mapsto L^2((s,t);X_q),\ \ \ \ \varphi\mapsto B(\cdot)\varphi,\] is bounded and onto. 
Therefore, the range of the operator $L$ defined in \eqref{eq:L} is still $(X_q,D_q)_{1/2, 2}$.

In the paper \cite{PS} it is proved that $U(t,s)$ maps $X_q$ into $(X_q,D_q)_{1/2,2}$ for $t>s$, and there exists $c>0$ such that 
\begin{equation}
\label{stima1/2,2}
\|U(t,s)x\|_{(X_q,D_q)_{1/2,2}} \leq \frac{c}{(t-s)^{1/2} } \|x\|_{X_q}, \quad 0\leq s<t\leq T, \; x\in X_q. 
\end{equation}

Now, for $s<t$ we split $U(t,s) = U(t, (t+s)/2)U((t+s)/2, s)$.   By the representation formula \eqref{reprU(t,s)} and estimate \eqref{stimak} we get
$$\|U((t+s)/2, s)  \|_{\mathcal L(X, L^\infty(\mathcal O))}\leq M_2^{1/2}|\mathcal O|  ((t-s)/2)^{-d/4}$$
while $\|U((t+s)/2, s)  \|_{\mathcal L(X)} \leq M$ for some $M$ independent of $t$, $s$. By interpolation we get 
$$\|U((t+s)/2, s)  \|_{\mathcal L(X, X_q)}\leq C_q(t-s)^{-(q-2)d/4q},  $$
for some $C_q$ independent of $t$, $s$. Using this estimate and 
  \eqref{stima1/2,2} with $s$ replaced by $(t+s)/2$, we get $U(t,s)(X) \subset (X_q,D_q)_{1/2,2}$, so that Hypothesis \ref{H2} is satisfied, and
$$\|U(t,s)\varphi \|_{(X_q,D_q)_{1/2,2}} \leq \frac{C}{(t-s)^{1/2 + (q-2)d/4q} } \|\varphi \|_{X}, \quad 0\leq s<t\leq T, \; \varphi \in X,$$
for some $C>0$ independent of $s$ and $t$,  which shows that  
Hypothesis \ref{Hyp:Lambda}  is satisfied with $\theta = 1/2 + (q-2)d/4q$.

\section{Acknowledgements}

The first  author was partially supported the NSF Research Grant DMS-1712934 (2017-2021), ``Analysis of Stochastic Partial Differential Equations with Multiple Scales" and NSF Research Grant DMS-1954299 (2020-2023), ``Multiscale Analysis of Infinite-Dimensional Stochastic Systems''. The second author is a member of GNAMPA-INDAM, and was partially supported by the Parma University Grant 2020 ``Topics in Deterministic and Stochastic Differential Equations". 

\end{document}